\documentclass{article}
\usepackage{expl3,xparse}      
\usepackage[table,dvipsnames]{xcolor} 
\usepackage{geometry}          
\usepackage{mfirstuc}          
\usepackage{settobox}          
\usepackage{enumitem}          
    \setlist[enumerate,1]{label = $\bullet$}
\usepackage[textwidth=16.3mm,textsize=scriptsize]{todonotes}
    \newcommand*{\newcommentator}[2][red!20]{%
        \expandafter\newcommand\csname #2\endcsname[1]{\todo[color=#1]{{\bf\makefirstuc{#2}:} ##1}}%
        \expandafter\newcommand\csname big#2\endcsname[1]{\todo[inline,color=#1]{{\bf\makefirstuc{#2}:} ##1}}}
    \newcommentator{pablo}

\usepackage{graphics,graphicx}
\usepackage{trimclip,adjustbox}

\usepackage{amssymb,amsmath}   
    \numberwithin{equation}{section} 
\usepackage{mathtools}         
\usepackage{mleftright}        
    \mleftright
\usepackage{stackrel}          
\usepackage{pgfcore}
\usepackage{pgfplots}
    \pgfplotsset{compat=1.18} 
\usepackage{tikz-cd}           
\usepackage{aliascnt}          

\usepackage{mathcro}
\usepackage{catcro}
\usepackage{mathenv}

\usepackage[pdfencoding=auto]{hyperref}
    \hypersetup{bookmarksopen=true, bookmarksnumbered=true, linktoc=all, hidelinks}
    
\usepackage{thmenv}

    \usetikzlibrary{calc}
    \usetikzlibrary{decorations.pathmorphing}
    \usetikzlibrary{spath3}

    \tikzset{curve/.style={settings={#1},to path={(\tikztostart)
        .. controls ($(\tikztostart)!\pv{pos}!(\tikztotarget)!\pv{height}!270:(\tikztotarget)$)
        and ($(\tikztostart)!1-\pv{pos}!(\tikztotarget)!\pv{height}!270:(\tikztotarget)$)
        .. (\tikztotarget)\tikztonodes}},
        settings/.code={\tikzset{quiver/.cd,#1}
            \def\pv##1{\pgfkeysvalueof{/tikz/quiver/##1}}},
        quiver/.cd,pos/.initial=0.35,height/.initial=0}

    \tikzset{between/.style n args={2}{/tikz/execute at end to={
        \tikzset{spath/split at keep middle={current}{#1}{#2}}
    }}}

    \tikzset{tail reversed/.code={\pgfsetarrowsstart{tikzcd to}}}
    \tikzset{2tail/.code={\pgfsetarrowsstart{Implies[reversed]}}}
    \tikzset{2tail reversed/.code={\pgfsetarrowsstart{Implies}}}
    \tikzset{no body/.style={/tikz/dash pattern=on 0 off 1mm}}

\let\oldsection\section
\makeatletter
\renewcommand\tableofcontents{%
    \oldsection*{\contentsname\@mkboth{\MakeUppercase\contentsname}{\MakeUppercase\contentsname}}\@starttoc{toc}%
}
\makeatother
\usepackage{mathdots}
\usepackage{bookmark,bbm}
\usepackage{float}

\title{Fully-Dualizable and Invertible $\mathcal{E}_n$-Algebras}
\author{Pablo Bustillo Vazquez}

\begin{document}

\maketitle
\begin{abstract}
    We prove a conjecture of Brochier, Jordan, Safronov, and Snyder \cite{BJSS21}, first formulated by Lurie \cite{Lur09a}, characterizing fully-dualizable and invertible $\mathcal{E}_n$-algebras viewed as objects in the higher Morita categories $\Mor_n(\V)$ \cite{Lur09a, Sch14, Hau17a, Hau23}.
    In other words, we characterize those $\mathcal{E}_n$-algebras which give rise to $(n+1)$-dimensional topological quantum field theories (TQFT), and those which give rise to invertible theories.
\end{abstract}
\tableofcontents
\newpage

\section{Introduction}

One of the central insights in the modern study of topological quantum field theory (TQFT) is that the structure of a field theory is governed by higher-categorical dualizability conditions.
This idea originates in the pioneering work of Baez-Dolan \cite{BD95}, who formulated the Cobordism Hypothesis.
They propose that an $n$-dimensional fully extended TQFT should be the data of a symmetric monoidal $(\infty,m)$-category $\cat{A}$ and a ``fully-$n$-dualizable'' object $a \inset \cat{A}$.
The initial such TQFT is then conjectured to be given by the $(\infty,n)$-category of framed bordisms. 
Typically, the symmetric monoidal $(\infty,m)$-categories $\cat{A}$ of interest arise from studying algebraic structures and the Cobordism Hypothesis can be seen as a duality between algebra and manifold topology.

Lurie observed in \cite[Sec. $4.1$]{Lur09a} that the collection of $\mathcal{E}_n$-algebras, in some fixed symmetric monoidal $(\infty, 1)$-category $\V$ with compatible colimits, can be organized in such a symmetric monoidal $(\infty, n+1)$-category $\Mor_n(\V)$ called the higher Morita category.
This generalizes the well-known classical Morita $(2,2)$-category whose objects are non-commutative rings and whose morphisms are bimodules.
The idea of this category was first sketched by Lurie in \cite[Def. $4.1.11$]{Lur09a} and later formalized by Scheimbauer and Haugseng in \cite{Sch14, Hau17, Hau23}.
It follows from the construction that all $k$-morphisms in $\Mor_n(\V)$, for $1 \leq k < n$, have left and right adjoints; this implies that any $\mathcal{E}_n$-algebra $A$ automatically gives rise to an $n$-dimensional fully-extended TQFT.
The computation of the partition function of this TQFT is exactly given by factorization homology with coefficients in $A$ as predicted by Lurie \cite[Thm. $4.1.24$]{Lur09a} and formalized in \cite{Sch14}.
However, the $n$-morphisms of $\Mor_n(\V)$ do not necessarily have adjoints, which raises the question of characterizing those $\mathcal{E}_n$-algebras $A$ for which the associated TQFT extends to $(n+1)$ dimensions, i.e., the fully-dualizable $\mathcal{E}_n$-algebras.
The purpose of this paper is to answer this question and further classify which $\mathcal{E}_n$-algebras form an invertible TQFT.

Classically, the fully-dualizable objects of the classical Morita $(2,2)$-category are known as smooth and proper algebras, and the invertible objects are known as Azumaya algebras.
These notions were generalized to the dg-setting by Toën \cite{Toe12}.
More explicitly, an algebra $A$ in $_{k}\Mod$ is smooth if $A$ is perfect as an $A^{op} \otimes_k A$-module, and proper if $A$ is perfect as a $k$-module.
A smooth and proper algebra $A$ is called an Azumaya algebra if the canonical
\begin{equation*}
    k \longeq \HC{1}(A) \\
    A^{\textnormal{op}} \otimes_k A \longeq \End_k(A)
\end{equation*}
are equivalences.

Lurie claimed in \cite[Rem. $4.1.27$]{Lur09a} that unwinding the proof of the Cobordism Hypothesis should give an explicit generalization, involving factorization homology, of the above conditions for full-dualizability of $\mathcal{E}_n$-algebras.
This criterion was later stated as a conjecture by Brochier, Jordan, Safronov, Snyder \cite[Conj. $1.8$]{BJSS21}, which we prove in more generality in \hyperref[thm:full_dualizability_criterion]{Theorem A} stated below using the language of factorization homology \cite{AFT17a, AFR18, AFT17}.

\begin{blank}[{\hyperref[thm:full_dualizability_criterion]{Thm. A}}]
    Let $\cat{V} \inset \EAlg{\infty}(\Cat{1}^\geom)$ be a symmetric monoidal $(\infty,1)$-category with geometric realizations and compatible tensor product.
    Let $0 \leq k < n$ and $M$ be a $k$-morphism in $\Mor_n(\cat{V})$.
    Then $M$ is fully-dualizable if and only if, for all $0 \leq i \leq n-k$, the canonical actions
    \begin{equation*}[rCl]
        \int_{(S^{i-1}, D^i)} (M, A) \curvearrowright &M& \\
        &M& \curvearrowleft \int_{(S^{i-1}, D^i)} (M, B)
    \end{equation*}
    exhibit $M$ as a dualizable module, where $A$ (resp. $B$) is the (co)domain of $M$ (see \autoref{constr:canonical_actions_thm_A} for details on these actions).
\end{blank}

One can interpret the factorizations $\int_{(S^{i-1}, D^i)} (M, A)$ as a higher version of Hochschild Homology of $M$ relative to $A$.
Brochier \textit{et al.} also conjectured a generalization of the characterization of Azumaya algebras \cite[Conj. $1.9$]{BJSS21}, which we prove in \hyperref[thm:invertibility_criterion]{Theorem B} stated below.

\begin{blank}[{\hyperref[thm:invertibility_criterion]{Thm. B}}]
    Let $\cat{V} \inset \EAlg{\infty}(\Cat{1}^\geom)$ be a symmetric monoidal $(\infty,1)$-category with geometric realizations and compatible tensor product.
    Let $0 \leq k < n$ and $M$ be a $k$-morphism in $\Mor_n(\cat{V})$.
    Then $M$ is invertible if and only if it is fully-dualizable and, for all $1 \leq i \leq n-k$, the universal $\mathcal{E}_{n-k-i+1}$-algebra maps
    \[
        \int_{(S^{i-1}, D^i)} (M, A) \longeq \HC{n-k-i}^B(M) \\
        \int_{(S^{i-1}, D^i)} (M, B) \longeq \HC{n-k-i}^A(M)
    \]
    are equivalences, where $A$ (resp. $B$) is the (co)domain of $M$ (see \autoref{constr:canonical_action_units} for details on these maps).
\end{blank}

\subsection{Overview}

\paragraph{Section 2}
To establish these results, we first need to prove a few preliminary results about $(\infty,n)$-categories with adjoints in \autoref{sec:adjoints}.
Most of these results are well-known to experts, but lack any rigorous proof in the literature, so we include them here for completeness.
All the work of this section heavily relies on the seminal result of Riehl and Verity \cite[Thm. 4.4.11]{RV16} bringing adjunctions in $(\infty,2)$-categories into a powerful homotopy coherent framework.
One of the most important of these folklore theorems is that the inclusion $\Cat{n}^{\textnormal{adj}} \mono \Cat{n}$ has left and right adjoints (see \autoref{prop:Free_Sub_exist}).
The existence of the left adjoint is trivial from the definition, but the existence of the right adjoint is more subtle.
One can construct it explicitly by hand, formalizing the intuition that it is given by the full sub-$(\infty,n)$-category on those morphisms which are fully-adjointable.
We choose here to construct this functor more formally, using the Adjoint Functor Theorem \cite[\href{https://kerodon.net/tag/06Q5}{Thm. 06Q5}]{kerodon} by proving that $\Cat{n}^{\textnormal{adj}} \mono \Cat{n}$ preserves colimits:

\begin{blank}[\autoref{prop:colim_adj}]
    Let $F \in I \to \Cat{n}$ be a diagram of $(\infty,n)$-categories which all admit right (resp. left) adjoints for all $k$-morphisms.
    Then so does $\Colim F$.
\end{blank}

The argument relies on the explicit way in which colimits of $(\infty,n)$-categories are constructed.
More explicitly, a $k$-morphism in some $\Colim_{i \inset \cat{I}} \cat{C}_i$ is always given by some intricate pasting of $k$-morphisms in the $\cat{C}_i$'s.
This is an intrinsic consequence of the ``Segalification'' construction for $(\infty,n)$-categories described in \cite{BS24}.
It is then well-known that composition of adjoints should have adjoints, which will in turn imply that $\Cat{n}^{\textnormal{adj}}$ is stable under colimits.
The main advantage of this approach is that it gives a general method for proving that some categories, defined as colimits, have adjoints.
We view this as a powerful tool for showing results of the form: ``to construct a [\textit{insert: higher notion of adjoint}], it suffices to construct [\textit{insert: specific family of usual adjunctions}]''.
A first statement of the type is given by the result below:
\begin{blank}[\autoref{prop:free_full_adj}]
    The free $(\infty, k+1)$-category on a single fully-adjointable $k$-morphism $\Free^{\{k\}}(c_k)$ is the colimit of the diagram
\[\begin{tikzcd}
	&&&& {\Free^{\{k\}}(c_k)} \\
	\\
	\cdots && {\Adj_k} && {\Adj_k} && {\Adj_k} && \cdots \\
	& {c_k} && {c_k} && {c_k} && {c_k}
	\arrow[curve={height=-12pt}, dashed, two heads, from=3-1, to=1-5]
	\arrow[curve={height=-6pt}, dashed, two heads, from=3-3, to=1-5]
	\arrow[dashed, two heads, from=3-5, to=1-5]
	\arrow[curve={height=6pt}, dashed, two heads, from=3-7, to=1-5]
	\arrow[curve={height=12pt}, dashed, two heads, from=3-9, to=1-5]
	\arrow["R"', two heads, from=4-2, to=3-1]
	\arrow["L", two heads, from=4-2, to=3-3]
	\arrow["R"', two heads, from=4-4, to=3-3]
	\arrow["L", two heads, from=4-4, to=3-5]
	\arrow["R"', two heads, from=4-6, to=3-5]
	\arrow["L", two heads, from=4-6, to=3-7]
	\arrow["R"', two heads, from=4-8, to=3-7]
	\arrow["L", two heads, from=4-8, to=3-9]
\end{tikzcd}\]
where all projections are epimorphisms, and $\Adj_k = \sigma^{k-1}(\Adj)$ is the walking adjoint $k$-morphism.
\end{blank}
This formalizes the principle that providing a ``once fully-adjoint'' $k$-morphism is equivalent to fitting it within a sequence of adjunctions.
The proof of the main \autoref{thm:invertibility_criterion} relies on a similar guiding principle: ``to show $f$ is fully-adjoint, it suffices to find adjoints for the iterated (co)units of the form $\eta(\epsilon(\epsilon(\ldots \eta(f)\ldots)))$'' (where we range over all words in $\epsilon$ and $\eta$ of some fixed length).

\paragraph{Section 3}
In the following section, we then focus on the geometric and algebraic tools required for our proofs.
On the geometric side, the main ingredient here is the theory of factorization homology of conically smooth stratified spaces developed by Ayala, Francis, Rozenblyum and Tanaka in \cite{AFT17a,AFR18,AFT17}.
For us, factorization homology is the natural tool to manipulate these $\mathcal{E}_k$-algebras in iterated bimodules $M$.
Following the guiding principle above, factorization homology will allow us to explicitly describe the algebraic structure corresponding to each of these iterated (co)units $\epsilon(\eta(\epsilon(\ldots \eta(M)\ldots)))$.

In \autoref{subsec:homological}, we introduce the notion of \textit{homological maps} and formulate a result to check when various ways of transferring algebraic structures between stratified spaces are equivalent.
In \autoref{subsec:handlebodies}, we define various stratified spaces which will ``carry'' this dualizability data, and various stratified maps which induce functors sending ``$M$'' to these ``$\eta(\eta(\epsilon(\ldots \epsilon(M)\ldots)))$''.
These spaces can be explicitly visualized as stratifications of the usual handlebodies.
Their algebraic structure can be visualized as iteratively ``pinching'' these handlebodies, through constructible bundle maps.

Finally, in \autoref{subsec:dual} and \autoref{subsec:deligne}, we prove a couple of technical algebraic results.
One of the main ingredients in our proof of \hyperref[thm:full_dualizability_criterion]{Theorem A} is a criterion (\autoref{prop:existence_dual}) for when $n$-morphisms in $\Mor_n(\V)$ have adjoints, essentially generalizing \cite[Prop. $4.6.2.13$]{Lur09a}.
Finally, to complete our proof of \hyperref[thm:invertibility_criterion]{Theorem B}, we also need to understand how Hochschild cohomology appears as duality data in $\Mor_n(\V)$, which is done in \autoref{prop:identifying_units}.

\paragraph{Section 4}
This short section is the compiled proofs of both our main Theorems.

\begin{notation}
    We will use the following notations and conventions.
    \begin{enumerate}
        \item For $n \geq 0$, we denote by $\Cat{n}$ the presentable $(\infty,1)$-category of $(\infty,n)$-categories.
        We have a sequence of reflections and coreflections:
    \[\begin{tikzcd}
        \Spa & {\Cat{1}} & {\Cat{2}} & \ldots & {\Cat{n}} & \ldots
        \arrow[""{name=0, anchor=center, inner sep=0}, hook, from=1-1, to=1-2]
        \arrow[""{name=1, anchor=center, inner sep=0}, "{{-}_{\leq 0}}", curve={height=-18pt}, from=1-2, to=1-1]
        \arrow[""{name=2, anchor=center, inner sep=0}, "{|-|_{>0}}"', curve={height=18pt}, from=1-2, to=1-1]
        \arrow[""{name=3, anchor=center, inner sep=0}, hook, from=1-2, to=1-3]
        \arrow[""{name=4, anchor=center, inner sep=0}, "{|-|_{>1}}"', curve={height=18pt}, from=1-3, to=1-2]
        \arrow[""{name=5, anchor=center, inner sep=0}, "{{-}_{\leq 1}}", curve={height=-18pt}, from=1-3, to=1-2]
        \arrow[""{name=6, anchor=center, inner sep=0}, hook, from=1-3, to=1-4]
        \arrow[""{name=7, anchor=center, inner sep=0}, "{|-|_{>2}}"', curve={height=18pt}, from=1-4, to=1-3]
        \arrow[""{name=8, anchor=center, inner sep=0}, "{{-}_{\leq 2}}", curve={height=-18pt}, from=1-4, to=1-3]
        \arrow[""{name=9, anchor=center, inner sep=0}, hook, from=1-4, to=1-5]
        \arrow[""{name=10, anchor=center, inner sep=0}, "{{-}_{\leq n-1}}", curve={height=-18pt}, from=1-5, to=1-4]
        \arrow[""{name=11, anchor=center, inner sep=0}, "{|-|_{>n-1}}"', curve={height=18pt}, from=1-5, to=1-4]
        \arrow[""{name=12, anchor=center, inner sep=0}, hook, from=1-5, to=1-6]
        \arrow[""{name=13, anchor=center, inner sep=0}, "{{-}_{\leq 0}}", curve={height=-18pt}, from=1-6, to=1-5]
        \arrow[""{name=14, anchor=center, inner sep=0}, "{|-|_{>n}}"', curve={height=18pt}, from=1-6, to=1-5]
        \arrow["\dashv"{anchor=center, rotate=-88}, draw=none, from=0, to=1]
        \arrow["\dashv"{anchor=center, rotate=-92}, draw=none, from=2, to=0]
        \arrow["\dashv"{anchor=center, rotate=-90}, draw=none, from=3, to=5]
        \arrow["\dashv"{anchor=center, rotate=-90}, draw=none, from=4, to=3]
        \arrow["\dashv"{anchor=center, rotate=-94}, draw=none, from=6, to=8]
        \arrow["\dashv"{anchor=center, rotate=-86}, draw=none, from=9, to=10]
        \arrow["\dashv"{anchor=center, rotate=-86}, draw=none, from=7, to=6]
        \arrow["\dashv"{anchor=center, rotate=-94}, draw=none, from=12, to=13]
        \arrow["\dashv"{anchor=center, rotate=-94}, draw=none, from=11, to=9]
        \arrow["\dashv"{anchor=center, rotate=-86}, draw=none, from=14, to=12]
    \end{tikzcd}\]

        \item For $n \geq 0$, we have an adjunction
    \[\begin{tikzcd}
        {\mathbf{Cat}_{(n,n)}} & {\Cat{n}}
        \arrow[curve={height=6pt}, hook, from=1-1, to=1-2]
        \arrow["{\ho_n}"', curve={height=6pt}, from=1-2, to=1-1]
    \end{tikzcd}\]

        \item We denote by $c_k$ the walking $k$-morphism, which we can see as an $(\infty, n)$-category for all $n \geq k$.
        We denote by $\partial c_k$ the boundary of the walking $k$-morphism, i.e., the walking pair of parallel $(k-1)$-morphisms, which can also be seen as an $(\infty,n)$-category for all $n \geq k-1$.
        It is the colimit of the diagram
    \[\begin{tikzcd}
        {c_0} & {c_1} & \cdots & {c_{k-1}} \\
        &&&& {\partial c_k} \\
        {c_0} & {c_1} & \cdots & {c_{k-1}}
        \arrow[from=1-1, to=1-2]
        \arrow[dashed, from=1-1, to=2-5]
        \arrow[from=1-1, to=3-2]
        \arrow[from=1-2, to=1-3]
        \arrow[dashed, from=1-2, to=2-5]
        \arrow[from=1-2, to=3-3]
        \arrow[from=1-3, to=1-4]
        \arrow[from=1-3, to=3-4]
        \arrow[dashed, from=1-4, to=2-5]
        \arrow[from=3-1, to=1-2]
        \arrow[dashed, from=3-1, to=2-5]
        \arrow[from=3-1, to=3-2]
        \arrow[from=3-2, to=1-3]
        \arrow[dashed, from=3-2, to=2-5]
        \arrow[from=3-2, to=3-3]
        \arrow[from=3-3, to=1-4]
        \arrow[from=3-3, to=3-4]
        \arrow[dashed, from=3-4, to=2-5]
    \end{tikzcd}\]
        A $k$-morphism $f\in c_k \to \cat{C}$ is also denoted $f\in a_{k-1} \to b_{k-1} \in \ldots \in a_0 \to b_0 \in \cat{C}$ where $a_i, b_i \in c_i \to \cat{C}$ are the induced projections above. 

        \item We let $\Cat{1}^\geom$ denote the symmetric monoidal $(\infty,1)$-category whose:
        \begin{enumerate}
            \item \textbf{Objects} are $(\infty,1)$-categories admitting geometric realizations.
            \item \textbf{Morphisms} are functors preserving geometric realizations.
        \end{enumerate}
        The symmetric monoidal structure is given by \cite[Cor. $4.8.1.4$]{Lur17}.
        As noted by Lurie in \cite[Rem. $4.8.1.9$]{Lur17}, unwinding the definitions, we see that a symmetric monoid in $\Cat{1}^\geom$ is precisely the datum of a symmetric monoidal $(\infty,1)$-category with geometric realizations and a compatible tensor product, in the sense that:
        the tensor product commutes with geometric realizations in each variable separately.

        \item A refinement of stratified space will always be denoted by a dashed arrow $X \dashrightarrow Y$.
        A (weakly) constructible bundle will always be denoted by a double-headed arrow $X \epi Y$.
        An open embedding will always be denoted by a hooked arrow $X \mono Y$ or by $X \longeq Y$ if it is an isotopy equivalence.
        A closed inclusion will always be denoted by a hooked arrow $X \xmono[\textnormal{cl}] Y$ with the label ``cl''.

        \item We let $S^{-1} = \emptyset$ and $D^0 = \{\bullet\}$.
                
        \item We denote operads \textit{à la} Lurie \cite[Chap. $2$]{Lur17} by $\operad{O} \to \Fin^*$ with an underline.
        Their ``underlying category'', \textit{i.e.} the fiber $\operad{O}^{\{\bullet\}^+}$ over $\{\bullet\}^+$, is denoted by $\mathcal{O}$ without an underline.
    \end{enumerate}
\end{notation}

\subsection{Acknowledgments}

The author would like to thank his PhD supervisor John Francis for many helpful discussions.

\section{Adjoints in \texorpdfstring{$(\infty,n)$}{(∞, n)}-categories}\label{sec:adjoints}

Throughout this section, we establish a few ``folklore'' results about $(\infty,n)$-categories with adjoints, which are crucial for the proof of the main Theorems.
However expected these results may seem, we could not find any rigorous proof of them in the literature, so we include them here for completeness.
The reader can safely skip this section on a first reading, and come back to it when needed.

Our approach is summarized as follows.
To work with robust homotopically-coherent notion, we define fully-adjointable morphisms in a $(\infty,n+1)$-category $\cat{C}$ as the morphisms that appear in the image of $\Sub^{\{1, \ldots, n\}}(\cat{C}) \mono \cat{C}$, where $\Sub^{\{1, \ldots, n\}}(\cat{C})$ is the full sub-$(\infty,n+1)$-category with adjoints.
To make sense of this definition, we then need to show that $\Sub^{\{1, \ldots, n\}}(\cat{C})$ is well-defined, and that the functor $\Sub^{\{1, \ldots, n\}}(\cat{C}) \mono \cat{C}$ is a monomorphism (see \autoref{prop:Free_Sub_exist}), which is the content of the second \autoref{subsec:Free_Sub}.
One way to tackle this problem is to show that colimits of $(\infty,n)$-categories with adjoints have adjoints (see \autoref{prop:colim_adj}), which is the content of the first \autoref{subsec:colim_adj}.
The advantage of this strategy is that it pushes us to develop tools to decide when a colimit of $(\infty,n)$-categories has adjoints in general.
This becomes important to link our definition of fully-adjointable morphisms to the more classical one: has an adjoint, with (co)units having adjoints, with (co)units having adjoints, and so on.
This is the content of \autoref{cor:lifting_once_fully_dualizable}, the main result of \autoref{subsec:lifting_adj}.
Finally, we finish this section by proving that the classifying space of the free/walking full-adjunction is contractible, which is the content of \autoref{subsec:invertibility} and will be useful for our second main result, \autoref{thm:invertibility_criterion}.

Before heading into the content of this section, we first state the definition of an adjoint morphism and the celebrated Theorem of Riehl and Verity \cite[Thm. 4.4.11]{RV16}.

\begin{definition}[Adjoint]
    Let $n > k \geq 1$ and $\cat{C}\in \Cat{n}$.
    We say that $l\in a_{k-1} \to b_{k-1} \in \ldots \in a_0 \to b_0 \in \cat{C}$ is a left adjoint if its equivalence class in $\ho_2(\cat{C}(a_0,b_0)(a_1,b_1)\ldots(a_{k-2}, b_{k-2}))$ is a left adjoint.
\end{definition}

\begin{construction}[Free/Walking Adjunction]
Recall that Schanuel and Street \cite{SS86} defined combinatorially the free $(2,2)$-category $\Adj$ containing an adjunction $L \in \begin{tikzcd}
        + & -
        \arrow[""{name=0, anchor=center, inner sep=0}, curve={height=-6pt}, from=1-1, to=1-2]
        \arrow[""{name=1, anchor=center, inner sep=0}, curve={height=-6pt}, from=1-2, to=1-1]
        \arrow["\dashv"{anchor=center, rotate=-90}, draw=none, from=0, to=1]
    \end{tikzcd} \in R $.
This category was later studied in depth by Riehl and Verity \cite{RV16}, in a higher categorical context.
\end{construction}

The results of this paper rely greatly on the powerful Theorem below.

\begin{theorem}[Homotopy Uniqueness of Adjunctions {\cite[Thm. 4.4.11]{RV16}}]\label{thm:riehl_verity}
   Let $\cat{C}\in \Cat{2}$.
   Precomposition by $L,R \in c_k \to \Adj$ gives monomorphisms
    \[\begin{tikzcd}
        {[\Adj,\mathcal{C}]} && {[c_1, \mathcal{C}]}
        \arrow["{L^*}", curve={height=-6pt}, hook, from=1-1, to=1-3]
        \arrow["{R^*}"', curve={height=6pt}, hook', from=1-1, to=1-3]
    \end{tikzcd}\]
    selecting the connected components of left (resp. right) adjoint morphisms.
\end{theorem}

\subsection{Colimits of \texorpdfstring{$(\infty,n)$}{(∞, n)}-categories with adjoints have adjoints}\label{subsec:colim_adj}

Note that a $k$-morphism in a colimit of $(\infty,n)$-categories ought to be given as pasting of $k$-morphisms in the categories in the diagram.
Moreover, pastings of $k$-morphisms with adjoints should in turn have adjoints.
In this subsection, we formalize this intuition to prove in \autoref{prop:colim_adj} that colimits of $(\infty,n)$-categories with adjoints have adjoints.
For this purpose, we use the following equivalent definitions of $(\infty,n)$-categories as:
\begin{enumerate}
    \item Rezk's complete Segal $\Theta_n$-spaces \cite{Rez10}, inspired by Joyal's unpublished ideas \cite{Joy97},
    \item Barwick's $n$-fold complete Segal spaces \cite{Bar05}, inspired by Rezk's complete segal spaces (for $n=1$) \cite{Rez01},
\end{enumerate}
In \cite{BS20}, Barwick and Schommer-Pries proved that these two definitions are equivalent.
In \cite{Hau17a}, Haugseng further proved that the two theories coincide prior to univalent completion.
Combining both results, we get the diagram where left (resp. right) adjoints commute with the vertical equivalences, and where the last column is equivalent to $\Cat{n}$.
\begin{equation}\label{diag:theta_delta}\begin{tikzcd}
	{\Psh(\Delta^n)} & {\Seg(\Delta^n)} & {\Seg^{\textnormal{univ}}(\Delta^n)} \\
	& {\Seg(\Theta^n)} & {\Seg^{\textnormal{univ}}(\Theta^n)}
	\arrow[""{name=0, anchor=center, inner sep=0}, "{\mathbb{L}_S}", curve={height=-12pt}, from=1-1, to=1-2]
	\arrow[""{name=1, anchor=center, inner sep=0}, curve={height=-12pt}, hook', from=1-2, to=1-1]
	\arrow[""{name=2, anchor=center, inner sep=0}, "{\mathbb{L}_C^\Delta}", curve={height=-12pt}, from=1-2, to=1-3]
	\arrow["\sim"{marking, allow upside down}, from=1-2, to=2-2]
	\arrow[""{name=3, anchor=center, inner sep=0}, curve={height=-12pt}, hook', from=1-3, to=1-2]
	\arrow["\sim"{marking, allow upside down}, from=1-3, to=2-3]
	\arrow[""{name=4, anchor=center, inner sep=0}, "{\mathbb{L}_C^\Theta}", curve={height=-12pt}, from=2-2, to=2-3]
	\arrow[""{name=5, anchor=center, inner sep=0}, curve={height=-12pt}, hook', from=2-3, to=2-2]
	\arrow["\dashv"{anchor=center, rotate=-90}, draw=none, from=0, to=1]
	\arrow["\dashv"{anchor=center, rotate=-90}, draw=none, from=2, to=3]
	\arrow["\dashv"{anchor=center, rotate=-90}, draw=none, from=4, to=5]
\end{tikzcd}\end{equation}
Note that an explicit formula for the $\Delta^n$-Segalification $\mathbb{L}_S^\Delta$ was established in \cite[Thm. D]{BS24} by Barkan and Steinebrunner following the theory of necklaces developed by Dugger and Spivak \cite{DS11}.
Another formula was established independently for $n=1$ by Ayala and Francis \cite[Sec. $3$]{AF24}.
An explicit formula for $\Theta^n$-univalent completion $\mathbb{L}_C^\Theta$ was established in \cite[Prop.  $2.4$]{AF18}, generalizing the Rezk's original formula \cite{Rez01}.
Using these formul\ae{} and the fact that left adjoints preserve colimits, we can understand explicitly how the $k$-cells of a colimit of categories are built from the $k$-cells of the categories in the diagram.

\begin{construction}[Suspension]\label{const:suspension}
    Recall that we have a functor $\sigma \in \Psh(\Delta^n) \to \Psh(\Delta^{n+1})$ defined as follows:
    \[
        \sigma X_{a_1, \ldots, a_{n+1}} = \{\bullet_-\} \sqcup \bigsqcup_{i = 1}^{a_1} X_{a_2, \ldots, a_{n+1}} \sqcup \{\bullet_+\}
    \]
    where the maps are defined in the obvious way.
    Note that if $X$ is reduced/Segal (in the $i^\textnormal{th}$ direction)/univalent, then $\sigma X$ is reduced/Segal (in the $(i+1)^\textnormal{th}$ direction)/univalent.
    We can iterate this construction and get
    \[
        \sigma^k X_{a_1, \ldots, a_{n+k}} = \left(\partial c_{k}\right)_{a_1, \ldots, a_{k-1}, 0, \ldots, 0} \bigsqcup \left( X_{a_{k+1}, \ldots, a_{n+k}} \times \prod_{i = 1}^{k} [a_i] \right)
    \]
    Note that $\sigma^k c_l = c_{l+k}$.
    We also have an adjunction
    \[\begin{tikzcd}
        {\Cat{n}} & {{}_{\partial c_k\backslash}\Cat{n+k}}
        \arrow[""{name=0, anchor=center, inner sep=0}, "{\sigma^k}"', curve={height=12pt}, hook, from=1-1, to=1-2]
        \arrow[""{name=1, anchor=center, inner sep=0}, "\hom"', curve={height=12pt}, from=1-2, to=1-1]
        \arrow["\dashv"{anchor=center, rotate=90}, draw=none, from=0, to=1]
    \end{tikzcd}\]
    where for $a_{k-1}, b_{k-1} \in \ldots \in a_0 \to b_0 \in \cat{C}$, we denote the corresponding $\hom(\partial c_k \to \cat{C}) = \cat{C}(a_0,b_0)(a_1,b_1)\ldots(a_{k-2}, b_{k-2})$.
\end{construction}

Denote by $\Adj_k = \sigma^{k-1} \Adj$ the free $(\infty, k+1)$-category containing an adjoint $k$-morphism.
We have an immediate corollary of \autoref{thm:riehl_verity}.

\begin{corollary}[Homotopy Uniqueness of (higher) Adjunctions {\cite[Thm. 4.4.11]{RV16}}]\label{cor:riehl_verity}
   Let $1 \leq k < n$ and $\cat{C}\in \Cat{n}$.
   Precomposition by $L,R \in c_k \to \Adj_k$ gives monomorphisms
    \[\begin{tikzcd}
        {[\Adj_k,\mathcal{C}]} && {[c_k, \mathcal{C}]}
        \arrow["{L^*}", curve={height=-6pt}, hook, from=1-1, to=1-3]
        \arrow["{R^*}"', curve={height=6pt}, hook', from=1-1, to=1-3]
    \end{tikzcd}\]
    selecting the connected components of left (resp. right) adjoint $k$-morphisms.
\end{corollary}
\begin{proof}
    $\sigma^{k-1}$ is a left adjoint composed with a forgetful from a slice, so it preserves contractible colimits and thus epimorphisms.
    Additionally, the notion of adjoint $k$-morphism is a suspension of that of adjoint $1$-morphisms.
\end{proof}

To formulate the next few lemmas, we now extend the notion of ``having adjoints'' to $n$-fold simplicial spaces.
Note that the homotopy uniqueness results above do not apply to general $n$-fold simplicial spaces, only to Segal and univalent/complete ones.
In other words, $c_k \to \Adj_k$ is not an epimorphism in $\Psh(\Delta^n)$ or $\Seg(\Delta^n)$ for $k < n$.

\begin{definition}[$n$-Fold Simplicial Space with Adjoints]\label{def:not_segat_adj}
    Let $1 \leq k < n$.
    Let $X \inset \Psh(\Delta^n)$ be a reduced $n$-fold simplicial space.
    We say that $X$ has right (resp. left) adjoints for all $k$-cells if for all $c_k \to X$, there exists a (not necessarily unique) lift
\[\begin{tikzcd}
	{c_k} & X \\
	{\Adj_k}
	\arrow[from=1-1, to=1-2]
	\arrow[from=1-1, to=2-1]
	\arrow[dashed, from=2-1, to=1-2]
\end{tikzcd}\]
    where the downward functor is $L$ (resp. $R$).
\end{definition}

In the next three results, we describe how Segalification interacts with having adjoints.
The first lemma below and its immediate corollary are the key technical tools for the inductive argument in \autoref{lem:adj_segalification}.
Recall from \cite[Thm. D]{BS24} that, for $X$ reduced, the Segalification functor can be computed as $\mathbb{L}^\Delta_S = \mathbb{L}^\Delta_{S,n} \circ \ldots \circ \mathbb{L}^\Delta_{S,1}$.

\begin{lemma}\label{lem:suspension_segal}
    Let $1 \leq i$ and $X$ be a $n$-fold Segal space.
    Then $\mathbb{L}^\Delta_{S,i} \left(\sigma^{i}X \bigsqcup_{c_{i-1}}\sigma^{i}X\right)$ is a $(n+i+1)$-fold Segal space.
\end{lemma}
\begin{proof}
    It suffices to prove the result for $i = 1$ as Segalification is seen to commute with suspension.
    By \autoref{const:suspension}, we have
    \[
        \left(\sigma X \bigsqcup_{c_{0}}\sigma X\right)_{a_1, \ldots, a_{n+1}} = \{\bullet_0\} \sqcup \bigsqcup_{i = 1}^{a_1} X_{a_2, \ldots, a_{n+1}} \sqcup \{\bullet_1\} \sqcup \bigsqcup_{i = 1}^{a_1} X_{a_2, \ldots, a_{n+1}} \sqcup \{\bullet_2\}
    \]
    By inspection, using the Necklace formula \cite[Thm. A]{BS24}, we get
    \begin{equation*}[rCr]
        \left(\mathbb{L}^\Delta_{S,i} \left(\sigma X \bigsqcup_{c_{0}}\sigma X\right)\right)_{a_1, \ldots, a_{n}} &=& \{\bullet_0\} \sqcup \bigsqcup_{i = 1}^{a_1} X_{a_2, \ldots, a_{n+1}} \sqcup \{\bullet_1\} \sqcup \bigsqcup_{i = 1}^{a_1} X_{a_2, \ldots, a_{n+1}} \sqcup \{\bullet_2\} \\
                && \bigsqcup_{1 \leq i < j \leq a_1} X_{a_2, \ldots, a_{n+1}}^2
    \end{equation*}
    This is easily seen to be Segal.
\end{proof}

\begin{corollary}[Composition of Adjoints]\label{cor:compose_adj}
    Let $1 \leq i < k < n$.
    There are lifts
\[\begin{tikzcd}
	{\Adj^k} & {c_k} && {\Adj^k} & {c_k} \\
	{\mathbb{L}^\Delta_{S,i} \left(\Adj^k \bigsqcup_{c_{i-1}} \Adj^k\right)} & {c_k \bigsqcup_{c_{i-1}} c_k} && {\mathbb{L}^\Delta_{S,k+1} \circ \mathbb{L}^\Delta_{S,k} \left(\Adj^k \bigsqcup_{c_{k-1}} \Adj^k\right)} & {c_k \bigsqcup_{c_{k-1}} c_k}
	\arrow[dashed, from=1-1, to=2-1]
	\arrow["L"', from=1-2, to=1-1]
	\arrow[dashed, from=1-4, to=2-4]
	\arrow["L", from=1-5, to=1-4]
	\arrow[from=2-2, to=1-2]
	\arrow[from=2-2, to=2-1]
	\arrow[from=2-5, to=1-5]
	\arrow[from=2-5, to=2-4]
\end{tikzcd}\]
\end{corollary}
\begin{proof}
    By \autoref{lem:suspension_segal} above, $\mathbb{L}^\Delta_{S,i} \left(\Adj^k \bigsqcup_{c_{i-1}} \Adj^k\right)$ is Segal.
    Note that $\mathbb{L}^\Delta_{S,k+1} \circ \mathbb{L}^\Delta_{S,k} \left(\Adj^k \bigsqcup_{c_{k-1}} \Adj^k\right)$ is also automatically Segal by \autoref{const:suspension}.
    The only equivalences in these $n$-fold simplicial sets are the identities, hence they are also complete.
    By \autoref{cor:riehl_verity}, it then suffices to show in classical (strict) $(n,n)$-category theory that an $i^\textnormal{th}$-composition of $k$-adjoints is an adjoint.
    This is a standard result from ordinary category theory.
\end{proof}

Using the two results above, we can now prove this key technical lemma below to handle adjoints in the Segalification of a reduced $n$-fold simplicial space.

\begin{lemma}[Segalification Preserves Adjoints]\label{lem:adj_segalification}
    Let $X$ be a reduced $n$-fold simplicial space that has right (resp. left) adjoints for all $k$-cells.
    Then its Segalification $\mathbb{L}^\Delta_S(X)$ has right (resp. left) adjoints for all $k$-cells.
\end{lemma}
\begin{proof}
    Using $\mathbb{L}^\Delta_S = \mathbb{L}^\Delta_{S,n} \circ \ldots \circ \mathbb{L}^\Delta_{S,1}$ for $X$ reduced, and using \cite[Lem. $2.15.a$]{BS24} which ensures $X$ remains reduced through these steps,
    one can observe that it is sufficient to prove the result separately:
    for $\mathbb{L}^\Delta_{S,i}$ with $1 \leq i \leq k-1$,
    and for $\mathbb{L}^\Delta_{S,k+1} \circ \mathbb{L}^\Delta_{S,k}$.

    \begin{enumerate}
        \item \textbf{Case $i < k$:}
        Fix $c_k = \Delta_{1,\ldots, 1, 0, \ldots, 0} \to \mathbb{L}_{S,i}^\Delta(X)$. By the Necklace formula \cite[Thm. A]{BS24} and by reducedness, we can find $m \geq 1$ and a diagram
\[\begin{tikzcd}
	{c_k} & {\Delta_{1,\ldots, 1, m, 1, \ldots, 1, 0, \ldots, 0}} & {\mathbb{L}_{S,i}^\Delta(X)} \\
	& {c_k \bigsqcup_{c_{i-1}} \ldots \bigsqcup_{c_{i-1}} c_k} & X
	\arrow[from=1-1, to=1-2]
	\arrow[from=1-2, to=1-3]
	\arrow["{\simeq_i}"', from=2-2, to=1-2]
	\arrow[from=2-2, to=2-3]
	\arrow[from=2-3, to=1-3]
\end{tikzcd}\]
    where the arrows marked by $\simeq_i$ are Segal equivalences for the $i^\textnormal{th}$-coordinate Segal conditions.
    Without loss of generality, it suffices to consider the case $m=2$.
    Using the universal property of $\mathbb{L}_{S,i}$ and using \autoref{cor:compose_adj}, we get a diagram
\[\begin{tikzcd}
	{c_k} &&& {\Adj_k} \\
	{\Delta_{1,\ldots, 1, 2, 1, \ldots, 1, 0, \ldots, 0}} && {\mathbb{L}_{S,i}\left(\Adj_k \bigsqcup_{c_i} \Adj_k\right)} \\
	{c_k \bigsqcup_{c_i} c_k} && {\Adj_k \bigsqcup_{c_i} \Adj_k} & {\mathbb{L}_{S,i}^\Delta(X)} \\
	&&& X
	\arrow["L", from=1-1, to=1-4]
	\arrow[from=1-1, to=2-1]
	\arrow["{\ref{cor:compose_adj}}"', from=1-4, to=2-3]
	\arrow[dashed, from=1-4, to=3-4]
	\arrow[from=2-1, to=2-3]
	\arrow[from=2-1, to=3-4]
	\arrow[dashed, from=2-3, to=3-4]
	\arrow["{\simeq_i}"', from=3-1, to=2-1]
	\arrow["{L \bigsqcup_{c_i} L}"{pos=0.8}, from=3-1, to=3-3]
	\arrow[from=3-1, to=4-4]
	\arrow["{\simeq_i}"', from=3-3, to=2-3]
	\arrow[dashed, from=3-3, to=4-4]
	\arrow["{\simeq_i}"', from=4-4, to=3-4]
\end{tikzcd}\]
    We want to show the top pentagon commutes.
    Because every other polytope commutes in the diagram, we have the desired lift.

        \item \textbf{Case $i = k$:} Fix $c_k = \Delta_{1,\ldots, 1, 0, \ldots, 0} \to \mathbb{L}^\Delta_{S,k+1} \circ \mathbb{L}^\Delta_{S,k}(X)$.
        Because $\mathbb{L}^\Delta_{S,k+1}$ doesn't affect the space of $k$-morphisms, we can factor this map through $c_k = \Delta_{1,\ldots, 1, 0, \ldots, 0} \to \mathbb{L}_{S,k}^\Delta(X)$.
        We then reproduce the same argument, keeping the first diagram identical, and replacing $\mathbb{L}_{S,i}^\Delta$ by $\mathbb{L}_{S,k+1}^\Delta \circ \mathbb{L}_{S,k}^\Delta$ everywhere in the second diagram.
    \end{enumerate}
\end{proof}

Sometimes, in situations such as \autoref{prop:free_full_adj}, one needs to show that the Segalification of a reduced $n$-fold simplicial space $X$ has adjoints even though $X$ itself does not strictly speaking have adjoints according to \autoref{def:not_segat_adj}.
Indeed, if any $k$-morphism of $X$ can be witnessed as a composition (or ``pasting'') of $k$-morphisms which do have adjoints, then the Segalification of $X$ will have adjoints for all $k$-morphisms.
For technical reasons, proving this for a general $k$ is impractical as the notion of ``composition'' can become quite complicated, and the notion of $n$-fold simplicial space doesn't quite capture the stability of this notion: a composition of compositions is not easily seen as a composition itself.
For the sake of this paper, it suffices to treat $k=1$, which is significantly more elementary.

\begin{lemma}[Partial Adjoints in Segalification]\label{lem:partial_adj_segal}
    Let $X$ be a reduced $n$-fold simplicial space.
    Let $\cat{I} \subset \pi_0 X_{1, 0, \ldots, 0}$ be the minimal subset such that:
    \begin{enumerate}
        \item if $f \in c_1 \to X$ has a right (resp. left) adjoint, as defined in \autoref{def:not_segat_adj}, $f$ in $\pi_0 X_{1,0, \ldots, 0}$ lies in $\cat{I}$,
        \item if $f \in c_1 \to X$ can be written as a composition of morphisms in $\cat{I}$, then $f$ lies in $\cat{I}$:
        in other words, if we have $\Delta_{r,0\ldots, 0} \epi c_1$ preserving end-points,
        and, if for all consecutive inclusions $c_1 \mono \Delta_{r,0\ldots, 0}$, the composition lies in $\cat{I}$,
        then $f$ lies in $\cat{I}$.
\[\begin{tikzcd}
	{c_1} \\
	{\Delta_{r,0\ldots, 0}} & X \\
	{c_1}
	\arrow["\forall"', from=1-1, to=2-1]
	\arrow["{\inset \cat{I}}", from=1-1, to=2-2]
	\arrow[from=2-1, to=2-2]
	\arrow[from=2-1, to=3-1]
	\arrow["f"', from=3-1, to=2-2]
\end{tikzcd}\]
    \end{enumerate}
    Suppose $\cat{I} = \pi_0 X_{1,0, \ldots, 0}$ is the whole set.
    Then $\mathbb{L}^\Delta_S(X)$ has right (resp. left) adjoints for all $1$-cells.
\end{lemma}
\begin{proof}
    The proof is almost identical to the second part of the proof above.
    One can simply observe with the same diagram that any $1$-cell in $\cat{I}$ has an adjoint in the Segalification. 
\end{proof}

Recall that our goal here is to show that a colimit of $(\infty, n)$-categories with adjoints also has adjoints.
As explained in the beginning of the section, this requires to show both a compatibility with Segalification (as done in the results above) and with completion, which is done below.

\begin{lemma}[Completion Preserves Adjoints]\label{lem:completion_lifts}
    Let $X$ be an $n$-fold Segal space.
    Let $c_k \to \mathbb{L}_C X$ be a $k$-cell in the completion of $X$.
    There exists a $k$-cell of $c_k \to X$ factoring $c_k \to \mathbb{L}_C X$.
\end{lemma}
\begin{proof}
    Note that by \eqref{diag:theta_delta}, one can prove this result for $\Theta^n$-spaces instead of $n$-fold Segal spaces.
    By \cite[Prop. $2.4$]{AF18}, we have an explicit formula for the univalent completion:
    \[
        \left[c_k, \mathbb{L}_C X\right] = \Colim_{\Theta_n} \left[c_k \times E(-), X \right]
    \]
    where $E \in \Theta_n \to \Psh(\Theta_n)$ sends $T \inset \Theta_n$ to the free strict $n$-groupoid generated by $T$ (see \cite[Def. $2.1$]{AF18} for more details).
    In particular, for any $c_k \to \mathbb{L}_C X$, we have a factorization
\[\begin{tikzcd}
	& {c_k} & {\mathbb{L}_CX} \\
	{c_k \times E(\bullet)} & {c_k \times E(T)} & X
	\arrow[from=1-2, to=1-3]
	\arrow[equals, from=2-1, to=1-2]
	\arrow["{\simeq_C}"', from=2-1, to=2-2]
	\arrow["{\simeq_C}", from=2-2, to=1-2]
	\arrow[from=2-2, to=2-3]
	\arrow["{\simeq_C}"', from=2-3, to=1-3]
\end{tikzcd}\]
\end{proof}

We can now finally prove the main result of this subsection:

\begin{proposition}[Colimits have Adjoints]\label{prop:colim_adj}
    Let $F \in I \to \Cat{n}$ be a diagram of $(\infty,n)$-categories which all admit right (resp. left) adjoints for all $k$-morphisms.
    Then so does $\Colim F$.
\end{proposition}
\begin{proof}
    By \eqref{diag:theta_delta}, one can compute the colimit as $\mathbb{L}_C \circ \mathbb{L}_S^\Delta \left(\Colim F\right)$ where the colimit is computed in $\Psh(\Delta^n)$.
    By definition, the collection of simplicial spaces with right (resp. left) adjoints for all $k$-morphisms is closed under colimits in $\Psh(\Delta^n)$.
    By \autoref{lem:adj_segalification}, $\mathbb{L}_S^\Delta \left(\Colim F\right)$ has right (resp. left) adjoints for all $k$-morphisms.
    Fix $c_k \to \mathbb{L}_C \circ \mathbb{L}_S^\Delta \left(\Colim F\right)$ a $k$-morphism.
    By \autoref{lem:completion_lifts}, we have a factorization through $c_k \to \mathbb{L}_S^\Delta \left(\Colim F\right)$ which then has an adjoint.
\end{proof}

\subsection{Freely Adding Adjoints and Subcategories with Adjoints}\label{subsec:Free_Sub}

Using the results of the previous subsection, we now establish some basic ``folklore'' results about the $(\infty, n)$-categories with adjoints, which we define below. 
The most important result of this section is \autoref{prop:Free_Sub_exist}.

\begin{definition}[Categories with Adjoints]
    Let $S \subset \{1,\ldots , n\}$ and $\A \in \Cat{n+1}$.
    We say that $\A$ has $S$-adjoints if, for all $k \inset S$ and all $c_k \to \A$, we have lifts
    \[\begin{tikzcd}
        {\Adj_k} \\
        {c_k} & \A \\
        {\Adj_k}
        \arrow[dashed, from=1-1, to=2-2]
        \arrow["L", from=2-1, to=1-1]
        \arrow[from=2-1, to=2-2]
        \arrow["R"', from=2-1, to=3-1]
        \arrow[dashed, from=3-1, to=2-2]
    \end{tikzcd}\]
    Denote by $\Cat{n+1}^{S\textnormal{-adj}} \mono \Cat{n+1}$ the full subcategory of $(\infty,n+1)$-categories with $S$-adjoints.
\end{definition}

Note that by the \autoref{cor:riehl_verity}, these lifts are unique if they exist.
Before we head into the proof of \autoref{prop:Free_Sub_exist}, we observe a couple of small results.
Note that by \autoref{cor:riehl_verity}, we find:

\begin{corollary}\label{cor:adj_fin_colim}
    Let $k \geq 0$.
    One can write $\Adj_k$ as a finite colimit diagram of the $c_i$'s.
\end{corollary}
\begin{proof}
    Existence of $L, R, \eta, \epsilon$ and the commutativity of the triangle identities can easily be encoded by some finite diagram.
    \autoref{cor:riehl_verity} then shows that the canonical cocone on $\Adj_k$ that one gets is colimiting.
\end{proof}

Using the results above, and our main result of the previous subsection (\autoref{prop:colim_adj}), we can show the existence of the free categories and full subcategories with adjoints.

\begin{proposition}[Free and Sub-Categories with Adjoints]\label{prop:Free_Sub_exist}
    Let $n \geq 0$ and $S \subset \{1,\ldots , n\}$.
    The inclusion $\Cat{n+1}^{S\textnormal{-adj}} \mono \Cat{n+1}$ fits into a bireflection between locally finitely-presentable categories:
\[\begin{tikzcd}
	{\Cat{n+1}^{S\textnormal{-adj}}} && {\Cat{n+1}}
	\arrow[""{name=0, anchor=center, inner sep=0}, hook, from=1-1, to=1-3]
	\arrow[""{name=1, anchor=center, inner sep=0}, "{\Free^{S}}"', curve={height=18pt}, from=1-3, to=1-1]
	\arrow[""{name=2, anchor=center, inner sep=0}, "{\SubAdj^{S}}", curve={height=-18pt}, from=1-3, to=1-1]
	\arrow["\dashv"{anchor=center, rotate=-99}, draw=none, from=0, to=2]
	\arrow["\dashv"{anchor=center, rotate=-81}, draw=none, from=1, to=0]
\end{tikzcd}\]
    Furthermore, the unit of the top adjunction $\cat{C} \epi \Free^{S}(\cat{C})$ is an epimorphism and the counit of the bottom adjunction $\SubAdj^{S}(\cat{C}) \mono \cat{C}$ is a monomorphism.
\end{proposition}
\begin{proof}
    $\Cat{n+1}$ has small colimits and is locally small.
    Note that $\Cat{n+1}^{S\textnormal{-adj}}$ is the category of $\cat{W}$-weakly local objects \cite[\href{https://kerodon.net/tag/04LH}{Def. 04LH}]{kerodon} for the small collection $\cat{W} = \{c_k \xepi[L] \Adj_k \mid k \inset S\} \cup \{c_k \xepi[R] \Adj_k \mid k \inset S\}$.
    The domains of these maps of $\cat{W}$ are all $c_k$ and hence compact. 
    The codomains of these maps are all $\Adj_k$'s and hence compact by \autoref{cor:adj_fin_colim}.
    By the Small Object Argument \cite[\href{https://kerodon.net/tag/04MD}{Cor. 04MD}]{kerodon}, the inclusion $\Cat{n+1}^{S\textnormal{-adj}} \mono \Cat{n+1}$ has a left adjoint $\Free^{S}$.
    
    Note that this implies $\Cat{n+1}^{S\textnormal{-adj}}$ has all colimits and that they are preserved by $\Free^{S}$.
    Explicitly, these colimits are computed by taking the diagram in $\Cat{n+1}$ via the inclusion, taking its colimit there, and then applying $\Free^{S}$.
    By \autoref{prop:colim_adj}, the last step is not necessary: the colimit taken in $\Cat{n+1}$ already has $S$-adjoints.
    Hence, the inclusion preserves all small colimits and thus is accessible.
    This shows that $\Cat{n+1}^{S\textnormal{-adj}}$ is a locally finitely-presentable category.
    By the Adjoint Functor Theorem \cite[\href{https://kerodon.net/tag/06Q5}{Thm. 06Q5}]{kerodon}, the inclusion $\Cat{n+1}^{S\textnormal{-adj}} \mono \Cat{n+1}$ has a right adjoint $\SubAdj^{S}$.

    Because epimorphisms are stable under pushout and transfinite composition, the unit $\cat{C} \epi \Free^{S}(\cat{C})$ is an epimorphism by construction.
    For $\A, \B \inset \Cat{n+1}$, we then have $[\A, \SubAdj^{S}(\B)] \simeq [\Free^{S}(\A),\SubAdj^{S}(\B)] \simeq [\Free^{S}(\A), \B] \mono [\A, \B]$, showing that the counit $\SubAdj^{S}(\B) \mono \B$ is a monomorphism.
\end{proof}

As a corollary to \autoref{prop:Free_Sub_exist}, we can now define the \textit{property} for morphisms of being fully-adjoint.

\begin{definition}[Fully-Adjoint Morphisms]
    Let $n \geq k \geq 0$, $S \subset \{k+1,\ldots, n\}$, $\cat{C} \inset \Cat{n+1}$ and $f \in c_k \to \cat{C}$ a $k$-morphism.
    We say that $f$ is fully-$S$-adjoint if $f$ factors as $c_k \to \SubAdj^{S}(\cat{C}) \mono \cat{C}$.
    If $S = \{k+1,\ldots, n\}$, we simply say that $f$ is fully-adjoint or fully-dualizable.
\end{definition}

One can easily build upon this and find a \textit{relative} version of the result above.

\begin{corollary}\label{cor:partial_Free_Sub_exists}
    Let $S \subset K \subset \{1,\ldots , n\}$.
    The inclusion $\Cat{n+1}^{K\textnormal{-adj}} \mono \Cat{n+1}^{S\textnormal{-adj}}$ fits into a bireflection between locally finitely-presentable categories:
\[\begin{tikzcd}
	{\Cat{n+1}^{K\textnormal{-adj}}} && {\Cat{n+1}^{S\textnormal{-adj}}}
	\arrow[""{name=0, anchor=center, inner sep=0}, hook, from=1-1, to=1-3]
	\arrow[""{name=1, anchor=center, inner sep=0}, "{\Free^{S \to K}}"', curve={height=18pt}, from=1-3, to=1-1]
	\arrow[""{name=2, anchor=center, inner sep=0}, "{\SubAdj^{S\to K}}", curve={height=-18pt}, from=1-3, to=1-1]
	\arrow["\dashv"{anchor=center, rotate=-91}, draw=none, from=0, to=2]
	\arrow["\dashv"{anchor=center, rotate=-89}, draw=none, from=1, to=0]
\end{tikzcd}\]
    Furthermore, the unit of the top adjunction $\cat{C} \epi \Free^{S\to K}(\cat{C})$ is an epimorphism and the counit of the bottom adjunction $\SubAdj^{S \to K}(\cat{C}) \mono \cat{C}$ is a monomorphism.
\end{corollary}
\begin{proof}
    The argument is entirely similar to that of \autoref{prop:Free_Sub_exist}.
    We can identify $\Cat{n+1}^{K\textnormal{-adj}} \mono \Cat{n+1}^{S\textnormal{-adj}}$ with the full subcategory of $\cat{W}$-weakly local objects for the small collection $\cat{W} = \{\Free^{S}(c_k \xepi[L] \Adj_k) \mid k \inset K\} \cup \{\Free^{S}(c_k \xepi[R] \Adj_k) \mid k \inset K\}$.
    The functor $\Free^{S}$ preserves compact objects (its right adjoint preserves filtered colimit) and epimorphisms (as it preserves pushouts), so we can apply the same argument as above.
\end{proof}

We now focus on a couple of results which are crucial to the inductive formulation of the main theorems of this paper. 

\begin{lemma}[$\Free$ preserves (some) Ajoints]\label{lem:Free_pres_adj}
    Let $S \subset \{2,\ldots , n\}$ and $m < \min(S)$.
    Then $\Free^{S}$ sends $\Cat{n+1}^{\{m\}\textnormal{-adj}}$ to $\Cat{n+1}^{S\cup\{m\}\textnormal{-adj}}$.
\end{lemma}
\begin{proof}
    Let $\cat{C} \inset \Cat{n+1}^{\{m\}\textnormal{-adj}}$.
    Recall from \autoref{prop:Free_Sub_exist} that the reflection $\cat{C} \epi \Free^{S}(\cat{C})$ is a transfinite composition of $\kappa$-many pushouts of maps in $\{c_k \xto[L] \Adj_k \mid k \inset S\} \cup \{c_k \xto[R] \Adj_k \mid k \inset S\}$, for some ordinal $\kappa$.
    We show by transfinite induction that the codomain of any such composition lands in $\Cat{n+1}^{\{m\}\textnormal{-adj}}$.
    \begin{enumerate}
        \item \textbf{Initialization:} By definition, $\cat{C}_0 = \cat{C} \inset \Cat{n+1}^{\{m\}\textnormal{-adj}}$.
        
        \item \textbf{Successor step:} Let $\alpha < \kappa$.
        There is a $k \inset S$ and a pushout diagram
        \[\begin{tikzcd}
            {c_k} & {\Adj_k} \\
            {\cat{C}_\alpha} & {\cat{C}_{\alpha+1}}
            \arrow[from=1-1, to=1-2]
            \arrow[from=1-1, to=2-1]
            \arrow[dashed, from=1-2, to=2-2]
            \arrow[dashed, from=2-1, to=2-2]
            \arrow["\lrcorner"{anchor=center, pos=0.125, rotate=180}, draw=none, from=2-2, to=1-1]
        \end{tikzcd}\]
        By \autoref{lem:completion_lifts}, it suffices to show the $m$-morphisms $c_m \to \mathbb{L}_{S}^\Delta(\cat{C}_{\alpha} \sqcup_{c_k} \Adj_k)$ have adjoints.
        It is immediate that for $a_1, \ldots, a_m \geq 0$, we have $(c_k)_{a_1, \ldots, a_m, 0, \ldots, 0} \simeq (\Adj_k)_{a_1, \ldots, a_m, 0, \ldots, 0}$.
        Hence, we have an equivalence $(\cat{C}_{\alpha})_{a_1, \ldots, a_m, 0, \ldots, 0} \simeq (\cat{C}_{\alpha} \sqcup_{c_k} \Adj_k)_{a_1, \ldots, a_m, 0, \ldots, 0}$.
        So the pushout is already Segal in the first $m$ directions and in the iterative Necklace formula \cite[Thm. D]{BS24}, the first $m$ steps do nothing.
        The last steps do not affect the space of $m$-morphisms either, so we have $(\cat{C}_{\alpha})_{a_1, \ldots, a_m, 0, \ldots, 0} \simeq (\mathbb{L}_{S}^\Delta(\cat{C}_{\alpha} \sqcup_{c_k} \Adj_k))_{a_1, \ldots, a_m, 0, \ldots, 0}$.
        By assumption, every $m$-morphism in $\cat{C}_{\alpha}$ has an adjoint, which concludes the successor step.

        \item \textbf{Limit step:} Let $\lambda \leq \kappa$ be a limit ordinal.
        We can write $\cat{C}_\lambda = \Colim_{\alpha < \lambda} \cat{C}_\alpha$, a colimit of categories with $m$-adjoints.
        By \autoref{prop:colim_adj}, we have $\cat{C}_\lambda \inset \Cat{n+1}^{\{m\}\textnormal{-adj}}$.
    \end{enumerate}
    By the transfinite induction principle, we get $\Free^{S}(\cat{C}) \inset \Cat{n+1}^{\{m\}\textnormal{-adj}} \cap \Cat{n+1}^{S\textnormal{-adj}} = \Cat{n+1}^{S\cup\{m\}\textnormal{-adj}}$.
\end{proof}

\begin{corollary}[Composing Free/Sub Constructions]\label{prop:free_sub_compose}
    Let $k < n$ and $S \subset \{k+1,\ldots , n\}$.
    Then $\Free^{\{k\} \cup S} = \Free^{S} \circ \Free^{\{k\}}$.
\end{corollary}
\begin{proof}
    Let $\A \inset \Cat{n+1}$.
    By \autoref{lem:Free_pres_adj}, we know that $\Free^{S} \circ \Free^{\{k\}}(\A)$ has $\{k\} \cup S$-adjoints.
    It then suffices to see that it satisfies the correct universal property.
    For $\B \inset \Cat{n+1}^{\{k\}\cup S\textnormal{-adj}}$, we have
    \[
        {[\Free^{S} \circ \Free^{\{k\}}(\A), \B]} \simeq [\Free^{\{k\}}(\A), \B] \simeq [\A, \B]
    \]
    as desired.
\end{proof}

\subsection{Lifting Once-Fully-Adjoint Morphisms}\label{subsec:lifting_adj}

In this subsection, we finalize the categorical prerequisites for the inductive argument at the core of the proof of the main theorem of this paper (\autoref{thm:full_dualizability_criterion}).
The result below formalizes the idea that ``constructing a full-$\{1\}$-adjoint is equivalent to finding a sequence of adjunctions''

\begin{proposition}[Free Full Adjunction]\label{prop:free_full_adj}
    Denote $\mathbb{Z}^{\to\leftarrow}$ the poset whose objects are the integers and which has morphisms $2n \to 2n\pm 1$ for all $n \in \mathbb{Z}$.
    We have a colimit diagram
\[\begin{tikzcd}
	&&&& {\Free^{\{k\}}(c_k)} \\
	\\
	\cdots && {\Adj_k} && {\Adj_k} && {\Adj_k} && \cdots \\
	& {c_k} && {c_k} && {c_k} && {c_k}
	\arrow[curve={height=-12pt}, dashed, two heads, from=3-1, to=1-5]
	\arrow[curve={height=-6pt}, dashed, two heads, from=3-3, to=1-5]
	\arrow[dashed, two heads, from=3-5, to=1-5]
	\arrow[curve={height=6pt}, dashed, two heads, from=3-7, to=1-5]
	\arrow[curve={height=12pt}, dashed, two heads, from=3-9, to=1-5]
	\arrow["R"', two heads, from=4-2, to=3-1]
	\arrow["L", two heads, from=4-2, to=3-3]
	\arrow["R"', two heads, from=4-4, to=3-3]
	\arrow["L", two heads, from=4-4, to=3-5]
	\arrow["R"', two heads, from=4-6, to=3-5]
	\arrow["L", two heads, from=4-6, to=3-7]
	\arrow["R"', two heads, from=4-8, to=3-7]
	\arrow["L", two heads, from=4-8, to=3-9]
\end{tikzcd}\]
We denote by $f^{\vee n}\in c_k \epi \Free^{\{k\}}(c_k)$ the $n$-th cocone projection (from left to right, meaning that $\vee$ corresponds to taking right adjoints), which are all epimorphisms.
\end{proposition}
\begin{proof}
   It suffices to show the result for $k=1$, as by \autoref{const:suspension}, suspension is the composition of a left adjoint and the forgetful from the slice, and hence commutes with contractible colimits.
   Firstly, we build the cocone diagram.
   By definition, we know that $\Free^{\{1\}}(c_1)$ is fully-$\{1\}$-adjoint and receives a universal functor $c_1 \to \Free^{\{1\}}(c_1)$.
   The space of cocones is the limit of the diagram $\mathbb{Z}^{\to\leftarrow} \to \Spa$ below:
\[\begin{tikzcd}
	\ldots && {[\Adj, \Free^{\{1\}}(c_1)]} && \cdots \\
	& {[c_1,\Free^{\{1\}}(c_1)]} && {[c_1, \Free^{\{1\}}(c_1)]}
	\arrow["{R^*}", from=1-1, to=2-2]
	\arrow["{L^*}"', from=1-3, to=2-2]
	\arrow["{R^*}", from=1-3, to=2-4]
	\arrow["{L^*}"', from=1-5, to=2-4]
\end{tikzcd}\]
    By \autoref{thm:riehl_verity}, this is an infinite zigzag of weak homotopy equivalences.
    By picking the point in the zero-th space corresponding to the universal functor $c_ 1 \to \Free^{\{1\}}(c_1)$, we get a canonical cocone diagram.

    Secondly, we build a functor $\Free^{\{1\}}(c_1) \to \Colim (\mathbb{Z}^{\to\leftarrow} \to \Cat{2})$.
    By universal property of $\Free^{\{1\}}(c_1)$, it is sufficient to show that $\Colim (\mathbb{Z}^{\to\leftarrow} \to \Cat{2})$ has $\{1\}$-adjoints and to pick a $1$-morphism $c_1 \to \Colim (\mathbb{Z}^{\to\leftarrow} \to \Cat{2})$.
    The first part is a consequence of \autoref{lem:partial_adj_segal} and \autoref{lem:completion_lifts}, as any morphism coming from a copy of $\Adj$ is given by a composition of $L$'s and $R$'s, which have adjoints.
    The second part follows by choosing the $1$-cell $c_1 \to \Colim (\mathbb{Z}^{\to\leftarrow} \to \Cat{2})$ given by the universal cocone projection on the zero-th copy of $c_1$.
    
    It is now easy to see that the two functors $\Free^{\{1\}}(c_1) \to \Colim (\mathbb{Z}^{\to\leftarrow} \to \Cat{2})$ and $\Colim (\mathbb{Z}^{\to\leftarrow} \to \Cat{2}) \to \Free^{\{1\}}(c_1)$ are inverses to each other by examining the universal properties.
\end{proof}

\begin{remark}
    This result can be generalized to construct any $\Free^{\{k, \ldots, n\}}(c_k)$ as a colimit.
    It requires, however, a more careful formulation of \autoref{prop:colim_adj}.
    One can even prove that for $k < n$, the colimit diagrams can be chosen finite, by carefully looking at the notion of redundancy of adjunction data (see \autoref{lem:redundancy} for an example of such a result).
\end{remark}

We get the following immediate corollary, which will be used in the proof of the main theorem of this paper (\autoref{thm:full_dualizability_criterion}).

\begin{corollary}[Lifting Once-Fully-Dualizable Morphisms]\label{cor:lifting_once_fully_dualizable}
    Let $1 \leq k \leq n$.
    Suppose we have a monomorphism $\cat{A} \mono \cat{B} \in \Cat{n+1}$, a map $\Free^{\{k\}}(c_k) \to \cat{B}$ such that all precomposition with $f^{\vee m} \in c_k \to \Free^{\{k\}}(c_k)$ lifts to $\cat{A}$ as in the diagram below.
    Then, there exists a dashed lift if and only if, for all $m$, there exist the dotted lifts:
\[\begin{tikzcd}
	& {\cat{A}} && {\cat{B}} \\
	{c_{k+1}} \\
	& {c_k} && {\Free^{\{k\}}(c_k)} \\
	{c_{k+1}}
	\arrow[hook, from=1-2, to=1-4]
	\arrow[dotted, from=2-1, to=1-2]
	\arrow["{\epsilon(f^{\vee m})}"{description}, from=2-1, to=3-4]
	\arrow[from=3-2, to=1-2]
	\arrow["{f^{\vee m}}"', two heads, from=3-2, to=3-4]
	\arrow[dashed, from=3-4, to=1-2]
	\arrow[from=3-4, to=1-4]
	\arrow[dotted, from=4-1, to=1-2]
	\arrow["{\eta(f^{\vee m})}"{description}, from=4-1, to=3-4]
\end{tikzcd}\]
    In other words: a fully-$\{k\}$-dualizable morphism lifts if all of its units and counits lift.  
\end{corollary}

\subsection{Invertibility}\label{subsec:invertibility}

In this final subsection, we compute the classifying spaces of free/walking adjunctions.
Contractibility of these spaces exactly corresponds to showing that any invertible morphism is fully-adjointable, which is part of the statement of \autoref{thm:invertibility_criterion}.

\begin{lemma}[Classifying Space of $\Adj$]\label{lem:adj_contractible}
    The localization of $\Adj$ by freely inverting $\eta$ and $\epsilon$ is the contractible category.
\end{lemma}
\begin{proof}
    In $\Cat{n}$, the terminal functor $c_1 \to \{\bullet\}$ can be seen as a localization or truncation.
    Indeed, using completeness, precomposing by this map induces the inclusion $[c_1, \cat{C}]_{\simeq} \mono [c_1, \cat{C}]$ of the connected component of all invertible $1$-morphisms (for example, see \cite[Thm. $6.2$ and Prop. $6.4$]{Rez01}). 
    We have a factorization $c_1 \epi \Adj \epi \Adj[\eta^{-1}, \epsilon^{-1}] \to \{\bullet\}$ where the composition is epimorphic.
    This implies that $\Adj[\eta^{-1}, \epsilon^{-1}] \epi \{\bullet\}$ is epimorphic as well and hence, by applying any $[-, \cat{C}]$, we get
    \[
        {[c_1, \cat{C}]_{\simeq} \mono [\Adj[\eta^{-1}, \epsilon^{-1}], \cat{C}] \mono [\Adj, \cat{C}] \mono [c_1, \cat{C}]}
    \]
    To show that $\Adj[\eta^{-1}, \epsilon^{-1}] \epi \{\bullet\}$ is an equivalence, it is sufficient to show that: given any adjunction $L \in \begin{tikzcd}
        a & b
        \arrow[""{name=0, anchor=center, inner sep=0}, curve={height=-6pt}, from=1-1, to=1-2]
        \arrow[""{name=1, anchor=center, inner sep=0}, curve={height=-6pt}, from=1-2, to=1-1]
        \arrow["\dashv"{anchor=center, rotate=-90}, draw=none, from=0, to=1]
    \end{tikzcd} \in R \in \cat{C}$, if $\eta$ and $\epsilon$ are invertible, then $L$ is invertible.
    Invertibility is a property detected by truncation to the homotopy $(1,1)$-category, where this is a classical result of ordinary category theory.
\end{proof}

Similarly, we get the following:

\begin{lemma}[Classifying Space of Higher Walking Adjoints]\label{lem:free_contractible}
    Let $1 \leq k \leq n$ and $S \subset \{k, \ldots, n\}$.
    The functor $\Free^S(c_k) \epi c_{k-1}$ is an epimorphism and exhibits $c_{k-1}$ as the localization at all $i$-morphisms for $i > k$, meaning that it is the reflection map in the localization $|-|_{> k} \in \Cat{n+1} \to \Cat{k}$.
\end{lemma}
\begin{proof}
    As above, the composition $c_k \epi \Free^S(c_k) \to c_{k-1}$ is epimorphic, implying the first part of the statement.
    Note that for $1 \leq l < k$, we have $\Free^{\{l\}}(c_k) = c_k \epi c_{k-1}$ and the result is a consequence of completeness.
    By \autoref{prop:free_sub_compose}, we can assume without loss of generality that $S = \{k\} \sqcup S'$ with $S' \subset \{k+1, \ldots, n\}$.
    We proceed by downward induction on $|S|$.
    \begin{enumerate}
        \item \textbf{Base case $|S| = 1$:} The truncation functor $|-|_{> k}$, being a left adjoint, will preserve the colimit in \autoref{prop:free_full_adj}.
        Because $|-|_{> k} \circ \sigma^{k-1} \simeq \sigma^{k-1} \circ |-|_{>1}$ and because of \autoref{lem:adj_contractible} above, the diagram is sent to the terminal diagram $n \inset \mathbb{Z}^{\to\leftarrow} \mapsto c_k$.
        As $\mathbb{Z}^{\to\leftarrow}$ is contractible, we get $|\Free^{\{k\}}(c_k)|_{> k} \simeq c_{k-1}$.

        \item \textbf{Inductive step:} Let $|S| > 1$.
        Similarly to the argument of \autoref{lem:adj_contractible}, we want to show that the inclusion $[c_k, \cat{C}]_{\simeq} \mono  [|\Free^S(c_k)|_{> k}, \cat{C}]$ is an equivalence.
        By \autoref{prop:free_sub_compose}, we have $[|\Free^S(c_k)|_{> k}, \cat{C}] \simeq [\Free^S(c_k), \cat{C}_{\leq k}] \simeq [\Free^{S'} \circ \Free^{\{k\}}(c_k), \cat{C}_{\leq k}] \simeq [\Free^{\{k\}}(c_k), \SubAdj^{S'}(\cat{C}_{\leq k})] \simeq [\Free^{\{k\}}(c_k), \cat{C}_{\leq k}] \simeq [c_k, \cat{C}_{\leq k}]_{\simeq}$ by the base case above.
    \end{enumerate}
\end{proof}

\section{Dualizability Data as Factorization Homology}\label{sec:factorization_homology}

The goal of this section is to establish the algebraic and geometric elements of the proof of our criteria.
In \autoref{subsec:homological} and \autoref{subsec:handlebodies}, we use factorization homology \cite{AFT17a, AFR18, AFT17} to construct the iterated ``(co)units of (co)units of (co)units of \ldots'', which we call adjunction data.
In \autoref{subsec:dual}, we prove a criterion for existence of adjoints of top-level morphisms in the higher Morita categories (see \cite{Lur09,Sch14, Hau17a, Hau23} for its definition).
Finally, in \autoref{subsec:deligne}, we identify some units in the dualizability data with universal algebra-map to some higher Hochschild Cohomology.

\begin{convention}[Fixing $\V$]
    Throughout the rest of this section, we fix $\V \inset \EAlg{\infty}\left(\Cat{1}^\geom\right)$, a symmetric monoidal $(\infty,1)$-category with geometric realizations and compatible tensor product.
\end{convention}

\subsection{Homological Commutative Diagrams}\label{subsec:homological}

This first subsection establishes a key technical tool to \textit{transfer} algebraic structures ``carried'' by different stratified spaces, which is stated in the form of \autoref{lem:commutative_diagrams}.
We start with a quick overview of the formalism of factorization homology of conically smooth stratified spaces developed by Ayala, Francis, Rozenblyum and Tanaka in \cite{AFT17, AFR18, AFT17a}.
All the stratified spaces we will consider will be conically smooth \cite[Subsec. $4.1$]{AFT17a} and finitary \cite[Def. $8.3.6$]{AFT17a}, so we omit these adjectives in the rest of this paper.
We slightly modify the original notations of \cite{AFT17a} as follows:
\begin{enumerate}
    \item $\disk$ will denote the symmetric monoidal $(1,1)$-category of ``disks'' with open embeddings: finite disjoint unions of basic stratified spaces, see \cite[Subsec. $3.2.$]{AFT17a} for their inductive definition,
    \item $\mfld$ will denote the symmetric monoidal $(1,1)$-category of stratified spaces with open embeddings,
    \item $\Disk$ will denote the symmetric monoidal $(\infty,1)$-category of disks with open embedding spaces see \cite[Subsec. $4.1$]{AFT17a} (denoted $\Snglr$ in the original text),
    \item $\Mfld$ will denote the symmetric monoidal $(\infty,1)$-category of stratified spaces with open embedding spaces (denoted $\Snglr$ in the original text).
\end{enumerate}

We now briefly recall the operadic notions required to talk about factorization homology of stratified spaces.
Once again, we slightly modify the original notations of \cite{AFT17} to fit our exposition.

\begin{definition}[Marked Operad]\label{def:marked_operad}
    An operad $\operad{O} \to \Fin_*$ is said to be \emph{marked} if it is endowed with a collection of coCartesian morphisms called \emph{pre-coCartesian} which contains all inert morphisms.
    For two marked operads $\operad{A}, \operad{B} \to \Fin_*$, the category of algebras $\Alg{\operad{A}}(\operad{B}) \mono [\operad{A}, \operad{B}]_{/\Fin^*}$ is the full-$(\infty,1)$-subcategory of functors over $\Fin^*$ on those functors which send pre-coCartesian morphisms to pre-coCartesian morphisms.
\end{definition}

Note that usual operads can be seen as marked operads by declaring only inert morphisms to be pre-coCartesian,
and symmetric monoidal $(\infty,1)$-categories can be seen as marked operads by declaring all coCartesian morphisms to be pre-coCartesian.
Recall from \cite[Subsubsection $2.4.3$]{Lur17} that given a category $\mathcal{O}$, we can form an operad $\mathcal{O}^\sqcup \to \Fin_*$ whose objects over $\mathbf{n}^+$ are $n$-tuples of objects of $\mathcal{O}$.
We can then directly adapt the following result from \cite{AFT17}.

\begin{lemma}[Pullback of Marked Operads along Right Fibrations, {\cite[Cor. $1.20$]{AFT17}}]\label{lem:pullback_marked_operad}
    Let $\operad{O} \to \Fin_*$ be a marked operad and $\mathcal{E} \to \mathcal{O}$ be a right fibration.
    The pullback
    \[\begin{tikzcd}
        {\operad{E}} & {\mathcal{E}^\sqcup} \\
        {\operad{O}} & {\mathcal{O}^\sqcup} \\
        & {\Fin^*}
        \arrow[dashed, from=1-1, to=1-2]
        \arrow[dashed, from=1-1, to=2-1]
        \arrow["\lrcorner"{anchor=center, pos=0.125}, draw=none, from=1-1, to=2-2]
        \arrow[from=1-2, to=2-2]
        \arrow[from=2-1, to=2-2]
        \arrow[from=2-1, to=3-2]
        \arrow[from=2-2, to=3-2]
    \end{tikzcd}\]
    exhibit $\operad{E}$ as a marked operad with pre-coCartesian morphisms those mapping to pre-coCartesian morphisms in $\operad{O}$.
\end{lemma}

\begin{definition}[{\cite[Const. $2.4$]{AFT17}}]\label{def:operads}
    Using \autoref{lem:pullback_marked_operad} above, for $X$ a stratified space, we define
    \begin{enumerate}
        \item $\udisk_{/X}$ as the marked $(1,1)$-operad obtained from the slice right fibration $\disk_{/X} \to \disk$ and the symmetric monoidal structure on $\disk$,
        \item $\umfld_{/X}$ as the marked $(1,1)$-operad obtained from the slice right fibration $\mfld_{/X} \to \mfld$ and the symmetric monoidal structure on $\mfld$,
        \item $\uDisk_{/X}$ as the marked $(\infty,1)$-operad obtained from the slice right fibration $\Disk_{/X} \to \Disk$ and the symmetric monoidal structure on $\Disk$,
        \item $\uMfld_{/X}$ as the marked $(\infty,1)$-operad obtained from the slice right fibration $\Mfld_{/X} \to \Mfld$ and the symmetric monoidal structure on $\Mfld$,
    \end{enumerate}
    To lighten notations, we let $\Alg{X}(\operad{O})$ denote $\Alg{\uDisk_{/X}}(\operad{O})$ for $\operad{O}$ a marked operad.
    More explicitly, the operad $\uDisk_{/X}$ is given as follows:
    \begin{enumerate}
        \item its objects lying over $\mathbf{n}^+$ are collections $(\bigsqcup_{1 \leq i \leq k_j} B_{i,j} \mono X)_{1 \leq j \leq n}$ of embeddings of a finite disjoint union of basics into $X$,
        \item the space of morphisms from $a = (\bigsqcup_{1 \leq i \leq k_j} B_{i,j} \mono X)_{1 \leq j \leq n}$ to $b = (\bigsqcup_{1 \leq i \leq l_j} B'_{i,j} \mono X)_{1 \leq j \leq m}$ over $\phi \in \mathbf{n}^+ \to \mathbf{m}^+ \in  \Fin_*$ of the pullback in spaces 
\[\begin{tikzcd}
	{\Disk(\B)_{/X}(a,b)} & {\{\bullet\}} \\
	{\prod_{1 \leq j' \leq m} \Emb(\bigsqcup_{j \inset \phi^{-1}(j')}\bigsqcup_{1 \leq i \leq k_j} B_{i,j}, \bigsqcup_{1 \leq i \leq l_{j'}} B'_{i,j'})} & {\prod_{j \inset \phi^{-1}(\mathbf{m})}\Emb(\bigsqcup_{1 \leq i \leq k_j} B_{i,j}, X)}
	\arrow[dashed, from=1-1, to=1-2]
	\arrow[dashed, from=1-1, to=2-1]
	\arrow[from=1-2, to=2-2]
	\arrow[""{name=0, anchor=center, inner sep=0}, from=2-1, to=2-2]
	\arrow["\lrcorner"{anchor=center, pos=0.125, rotate=45}, draw=none, from=1-1, to=0]
\end{tikzcd}\]
        \item a morphism above is pre-coCartesian if each of the maps $\bigsqcup_{j \inset \phi^{-1}(j')}\bigsqcup_{1 \leq i \leq k_j} B_{i,j} \mono \bigsqcup_{1 \leq i \leq l_{j'}} B'_{i,j'})$ in the bottom left space is an isotopy equivalence.
    \end{enumerate}

    Note that these constructions are functorial in $X$ with respect to open embeddings, as it directly follows from the slice construction which is itself functorial.
    In particular, we have an action $\Aut(X) \curvearrowright \uDisk_{/X}$ by equivalences of marked operads.
\end{definition}

Using the equivalence of \cite[Thm. $2.43$]{AFT17}, one can define homology theories as follows.

\begin{definition}[Homology Theories {\cite[Def. $2.37$]{AFT17}}]\label{def:homology_theory}
    For $X$ a stratified space, the category of \emph{homology theories} $\cat{H}(\uMfld_{/X}, \V)$ is the full-$(\infty,1)$-subcategory
    \[
        \cat{H}(\uMfld_{/X}, \V) \mono \Alg{\uMfld_{/X}}(\V)
    \]
    on those morphisms of operads $\cat{H}$ such that the following equivalent conditions are satisfied:
    \begin{enumerate}
        \item For all $Z \inset \Mfld_{/X}$, the canonical morphism
        \[
            \Colim\left( \Disk_{/Z} \to \Mfld_{/X} \to \V \right) \longeq \cat{H}(Z)
        \]
        is an equivalence.
        \item For all $Z \cong Z_- \cup_{\mathbb{R} \times Z_0} Z_+$ collar-gluing over $X$, the canonical morphism
        \[
            \left| \cat{H}(Z_-) \otimes \cat{H}(\mathbb{R} \times Z_0)^{\otimes \bullet} \otimes \cat{H}(Z_+) \right|_\bullet \longeq \cat{H}(Z)
        \]
        is an equivalence.
    \end{enumerate}
    Similarly, we define the category of \textit{pre-homology theories} as the full-$(\infty,1)$-subcategory $H(\umfld_{/X}, \V) \mono \Alg{\umfld_{/X}}(\V)$ on those algebras which satisfy the corresponding colimit conditions.
\end{definition}

In the original text of \cite{AFT17}, this definition is given for $\V$ a symmetric monoidal $(\infty,1)$-category with sifted colimits and whose monoidal structure distributes over sifted colimits.
It's a harmless generalization to replace sifted colimits by geometric realizations, as long as we only consider \textit{finitary} stratified spaces, meaning those obtained from basics via a sequence of collar-gluing \cite[Def. $8.3.6$]{AFT17a}, as is the case here.
Indeed, it is proven in \cite[Cor. $2.38$]{AFT17} that for a collar-gluing $X \cong X_- \cup_{\mathbb{R} \times X_0} X_+$, we have a geometric realization
\[
    \left| \Disk_{/X_-} \times \Disk_{/\mathbb{R} \times X_0}^\bullet \times \Disk_{/X_+ } \right|_\bullet \longeq \Disk_{/X}
\]
so that the collection of $\Disk_{/-}$-indexed colimits is generated by geometric realizations, and thus also the collection of $\disk_{/-}$-indexed colimits, as $\disk_{/X} \to \Disk_{/X}$ is a localization \cite[Prop. $2.22$]{AFT17} and hence final.

\begin{corollary}[Disk-Algebras are the coefficients of Homology Theories]\label{cor:algebra_homology}
    Let $X$ be a stratified space.
    By left Kan extension along the inclusions $\disk_{/X} \mono \mfld_{/X}$ and $\Disk_{/X} \mono \Mfld_{/X}$, we have the commutative diagram
\[\begin{tikzcd}
	{\Alg{X}(\V)} & {\cat{H}(\uMfld_{/X}, \V)} & {\Alg{\uMfld_{/X}}(\V)} \\
	\\
	{\Alg{\udisk_{/X}}(\V)} & {H(\umfld_{/X}, \V)} & {\Alg{\umfld_{/X}}(\V)}
	\arrow["\sim"{marking, allow upside down}, from=1-1, to=1-2]
	\arrow[hook, from=1-1, to=3-1]
	\arrow[hook, from=1-2, to=1-3]
	\arrow[hook, from=1-2, to=3-2]
	\arrow[from=1-3, to=3-3]
	\arrow["\sim"{marking, allow upside down}, from=3-1, to=3-2]
	\arrow[hook, from=3-2, to=3-3]
\end{tikzcd}\]
\end{corollary}
\begin{proof}
    This is a direct consequence of \cite[Lem. $2.16$ and $2.17$]{AFT17} in which the details of how left Kan extensions behave in the context of marked operads can be found.
\end{proof}

We can then define a convenient way to transfer between homology theories.
This definition is a small abstraction from the behavior of constructible bundles in \cite[Subsec. $2.5$]{AFT17}.

\begin{definition}[Homological Morphisms]\label{def:homological_morphism}
    Let $X,Y$ be stratified spaces together with a marked operad morphism $\alpha\in \umfld_{/Y} \to \umfld_{/X}$.
    We say that this morphism is \emph{homological} if
    \begin{enumerate}[label=(\roman*)]
        \item the functor $\disk_{/Y} \xto[\Disk_{/\alpha(-)}] \Cat{1}$ sends isotopy equivalences to equivalences and,
        \item the following equivalent conditions hold:
        \begin{enumerate}
            \item for all $V \inset \mfld_{/Y}$, the canonical functor below is an equivalence
                \[
                    \Colim \left( \disk_{/V} \xto[\Disk_{/\alpha(-)}] \Cat{1} \right) \longeq \Disk_{/\alpha(V)}
                \]
            \item for all $V \inset \disk_{/Y}$, the canonical projection $X_{V,\alpha} \xto[(\star)] \disk_{/\alpha(V)}$ from the \textit{comma} $(1,1)$-category below is $(\infty,1)$-final
                \[\begin{tikzcd}
                    {\disk_{/V}} & {\mfld_{/\alpha(V)}} \\
                    {X_{V,\alpha}} & {\disk_{/\alpha(V)}}
                    \arrow["\alpha", from=1-1, to=1-2]
                    \arrow[dashed, from=2-1, to=1-1]
                    \arrow["{(\star)}"', dashed, from=2-1, to=2-2]
                    \arrow[between={0.3}{0.7}, Rightarrow, from=2-2, to=1-1]
                    \arrow[hook', from=2-2, to=1-2]
                \end{tikzcd}\]
        \end{enumerate}
    \end{enumerate}
\end{definition}
In the notations of \cite[Subsec. $2.5$]{AFT17}, the comma category $X_{V,\alpha}$ would be described as the limit of the diagram below.
\[\begin{tikzcd}
	&& {X_{V,\alpha}} \\
	\\
	{\disk_{/\alpha(V)}} && {[c_1, \mfld_{/\alpha(V)}]} && {\disk_{/V}} \\
	& {\mfld_{/\alpha(V)}} && {\mfld_{/\alpha(V)}}
	\arrow["{(\star)}"', dashed, from=1-3, to=3-1]
	\arrow[dashed, from=1-3, to=3-3]
	\arrow[dashed, from=1-3, to=3-5]
	\arrow[hook', from=3-1, to=4-2]
	\arrow["\dom", from=3-3, to=4-2]
	\arrow["\cod"', from=3-3, to=4-4]
	\arrow[from=3-5, to=4-4]
\end{tikzcd}\]
\begin{proof}[of the Equivalence of Conditions]
    The proof follows from Joyal's $(\infty,1)$-categorical version of Quillen's Theorem A \cite[\href{https://kerodon.net/tag/02NY}{Thm. 02NY}]{kerodon}.
    Details can be found in the proof of \cite[Lem. $2.27$]{AFT17}, whose statement is essentially identical.
\end{proof}

The proof of \cite[Thm. $2.25$]{AFT17} can be adapted in our context to show the following, by using the ``universal'' algebra $A = \Disk_{/-}$ \cite[Cor. $2.38$]{AFT17}.

\begin{lemma}[Composition stays Homological]\label{lem:homological_composition}
    Homological morphisms are closed under composition.
\end{lemma}
\begin{proof}
    Let $X,Y,Z$ be stratified spaces and $\umfld_{/Z} \xto[\beta] \umfld_{/Y} \xto[\alpha] \umfld_{/X}$ be composable homological morphisms.  
    Let $W \inset \mfld_{/Z}$.
    Consider the diagram below.
\[\begin{tikzcd}
	{\disk_{/W}} & {\mfld_{/\beta(W)}} \\
	{X_{W,\beta}} & {\disk_{/\beta(W)}} & {\Cat{1}}
	\arrow["\beta", from=1-1, to=1-2]
	\arrow["{\Disk_{/-}}", from=1-2, to=2-3]
	\arrow[dashed, from=2-1, to=1-1]
	\arrow["{(\star)}"', dashed, from=2-1, to=2-2]
	\arrow[between={0.3}{0.7}, Rightarrow, from=2-2, to=1-1]
	\arrow[hook', from=2-2, to=1-2]
	\arrow["{\Disk_{/-}}"', from=2-2, to=2-3]
\end{tikzcd}\]
    As in the proof of \cite[Thm. $2.25$]{AFT17}, we note that the right square is a comma and the left triangle is a left Kan extension, implying that the outer triangle is also a left Kan extension.
    In particular, we have
    \[
        \Colim \left( \disk_{/W} \xto[\Disk_{/\beta(-)}] \Cat{1}\right) \revlongeq \Colim \left(X_{W,\beta} \to \disk_{/\beta(W)} \xto[\Disk_{/-}] \Cat{1}\right) \longeq \Colim\left(\disk_{/\beta(W)} \xto[\Disk_{/-}] \Cat{1}\right)
    \] 
    This gives the commutative diagram below, written with subscript notations to lighten the expressions.
    \[\begin{tikzcd}
        {\Colim_{W' \subseteq W}\left(\Colim_{V \subseteq \beta(W')} \Disk_{/\alpha(V)}\right)} & {\Colim_{W' \subseteq W} \Disk_{/\alpha\circ \beta(W')}} \\
        {\Colim_{V \subset \beta(W)} \Disk_{/\alpha(V)}} & {\Disk_{/\alpha \circ \beta(W)}}
        \arrow["\simeq", "{\alpha \textnormal{ hom.}}"', from=1-1, to=1-2]
        \arrow["\simeq"', "{\beta \textnormal{ hom.}}", from=1-1, to=2-1]
        \arrow[from=1-2, to=2-2]
        \arrow["\simeq", "{\alpha \textnormal{ hom.}}"', from=2-1, to=2-2]
    \end{tikzcd}\]
    where the horizontal functors are equivalences by property $(ii.a)$ for $\alpha$.
    Thus, the right vertical functor is an equivalence as well.
    We are left to prove that $\disk_{/Z} \xto[\Disk_{/\alpha\circ\beta(-)}] \Cat{1}$ sends isotopy equivalences to equivalences.
    If $W \mono W' \in \disk_{/Z}$ is an isotopy equivalence, then by property $(i)$ for $\beta$ the induced functor $\Disk_{/\beta(W)} \to \Disk_{/\beta(W')}$ is an equivalence.
    By property $(i)$ for $\alpha$, the colimits computed in the bottom left of the square above for $W$ (resp. $W'$) fit in the following diagram
    \[\begin{tikzcd}
        {\disk_{/\beta(W)}} & {\Disk_{/\beta(W)}} \\
        {\disk_{/\beta(W')}} & {\Disk_{/\beta(W)}} & {\Cat{1}}
        \arrow["{(\star)}", from=1-1, to=1-2]
        \arrow[from=1-1, to=2-1]
        \arrow["\sim"{marking, allow upside down}, from=1-2, to=2-2]
        \arrow["{\Disk_{/\alpha(-)}}", from=1-2, to=2-3]
        \arrow["{(\star)}"', from=2-1, to=2-2]
        \arrow["{\Disk_{/\alpha(-)}}"', from=2-2, to=2-3]
    \end{tikzcd}\]
    where the functors marked by $(\star)$ are final, and the one marked by $\sim$ is an equivalence.
    Thus, the induced morphism of colimits is an equivalence as well, and $\alpha \circ \beta$ is indeed homological.
\end{proof}

The following result justifies our naming convention for ``homological'' morphisms.

\begin{lemma}[Lifting Homology Theories]\label{lem:lifting_homology}
    Let $X, Y$ be stratified spaces and $\alpha\in \umfld_{/Y} \to \umfld_{/X}$ be homological.
    Then, we have a unique lift
\[\begin{tikzcd}
	{\cat{H}(\uMfld_{/X}, \V)} & {\cat{H}(\uMfld_{/Y}, \V)} \\
	{\Alg{\umfld_{/X}}(\V)} & {\Alg{\umfld_{/Y}}(\V)}
	\arrow[dotted, from=1-1, to=1-2]
	\arrow[hook, from=1-1, to=2-1]
	\arrow[hook, from=1-2, to=2-2]
	\arrow[from=2-1, to=2-2]
\end{tikzcd}\]
\end{lemma}
\begin{proof}
    If such a lift exists, it is necessarily unique as the vertical functors are fully-faithful and hence monomorphisms by \autoref{cor:algebra_homology}.
    To lift a $\cat{H} \inset \cat{H}(\uMfld_{/X}, \V)$, it suffices to check:
    \begin{enumerate}[label=(\alph*)]
        \item for $V \inset \mfld_{/Y}$, we have an equivalence $\Colim \left( \disk_{/V} \xto[H(\alpha(-))] \V \right) \longeq \cat{H}(\alpha(V))$,
        \item for $V \mono V' \in \disk_{/Y}$ an isotopy equivalence, the induced morphism $\cat{H}(\alpha(V)) \simeq \cat{H}(\alpha(V'))$ is an equivalence. 
    \end{enumerate}
    Similarly to the proof of \autoref{lem:homological_composition}, we have a pasting of Kan extensions
\[\begin{tikzcd}
	{\disk_{/V}} & {\mfld_{/\alpha(V)}} \\
	{X_{V,\alpha}} & {\disk_{/\alpha(V)}} & \V
	\arrow["\alpha", from=1-1, to=1-2]
	\arrow["{\cat{H}}", from=1-2, to=2-3]
	\arrow[dashed, from=2-1, to=1-1]
	\arrow["{(\star)}"', dashed, from=2-1, to=2-2]
	\arrow[between={0.3}{0.7}, Rightarrow, from=2-2, to=1-1]
	\arrow[hook', from=2-2, to=1-2]
	\arrow["{\cat{H}}"', from=2-2, to=2-3]
\end{tikzcd}\]
    which induces 
    \[
        \Colim \left( \disk_{/V} \xto[H(\alpha(-))] \V\right) \revlongeq \Colim \left(X_{V,\alpha} \to \disk_{/\alpha(V)} \xto[\cat{H}(-)] \V\right) \longeq \Colim\left(\disk_{/\alpha(V)} \xto[\cat{H}(-)] \V\right) \longeq \cat{H}(\alpha(V))
    \]
    proving $(a)$.
    Now $(b)$ follows directly from the formula above and from condition $(i)$ of \autoref{def:homological_morphism} which implies $\disk_{/\alpha(V)} \simeq \disk_{/\alpha(V')}$.
\end{proof}

We now describe four different classes of maps between stratified spaces which induce homological morphisms, and hence allow us to transfer algebraic structures between them.
We will denote these \textbf{homological maps}.
\begin{enumerate}
    \item \textbf{Open Embeddings:} As explained at the end of \autoref{def:operads}, open embeddings induce morphisms of marked operads and in particular, we get an action $\Aut(X) \curvearrowright \Disk_{/X}$.
    On $\pi_0$, this can be seen as acting through homological marked morphisms.

    \item \textbf{Refinements:} As explained in \cite[Lem. $2.23$]{AFT17}, refinements $\tilde{X} \dasharrow X$ of stratified spaces induce a localization $\uDisk_{/\tilde{X}} \dasharrow \uDisk_{/X}$ of marked operads.
    The same argument proves that for any $\tilde{V} \inset \mfld_{/\tilde{X}}$, the induced functor $\disk_{/\tilde{V}} \to \Disk_{/V}$ is a localization and hence final.
    In particular, the diagram in $(ii.a)$ in \autoref{def:homological_morphism} factors through this functor, and hence we get
    \[
        \Colim \left( \disk_{/\tilde{V}} \to \Disk_{/V} \xto[\Disk_{/-}] \Cat{1} \right) \longeq \Colim \left(\Disk_{/V} \xto[\Disk_{/-}] \Cat{1} \right) \longeq \Disk_{/V}
    \]
    by \cite[Cor. $2.38$]{AFT17}.
    Thus, refinements induce homological morphisms of marked operads.
    
    \item \textbf{Weakly Constructible Bundles:} Recall that a stratified map $\pi \in X \to Y$ is a \textit{weakly constructible bundle} if for each stratum $Y_q \subset Y$, the corestriction $\pi_{| \pi^{-1}(Y_q)}: \pi^{-1}(Y_q) \to Y_q$ is a fiber bundle (see \cite[Def. $3.6.1$]{AFT17a}).
    As explained in \cite[Lem. $2.24$ and Thm. $2.25$]{AFT17}, weakly constructible bundles of stratified spaces induce homological morphisms of marked operads by taking preimages.
    
    \item \textbf{Closed Strata Inclusions:} For $Y \xmono[\textnormal{cl}] X$ the inclusion of a closed union of strata, we have a canonical morphism of operads $\Disk_{/X} \to \Disk_{/Y}$ which ``forgets'' the algebraic data corresponding to the other strata.
    We can define this functor as follows.
    From \cite[Prop. $8.2.5$]{AFT17a}, we get a natural proper constructible bundle $\Link_Y(X) \xto[\pi] X$ inducing a cone construction
    \begin{equation}\label{eq:closed_inclusion}\begin{tikzcd}
        Y & {C(\pi)} & {C(\pi)_{\textnormal{unref.}}} & X
        \arrow[two heads, from=1-2, to=1-1]
        \arrow[dashed, from=1-2, to=1-3]
        \arrow[hook, from=1-3, to=1-4]
    \end{tikzcd}\end{equation}
    where the left map is a constructible bundle, the center one is a refinement and the right one is an open embedding.
    The construction of the right map can be seen as a stratified notion of a tubular neighborhood of $Y$ in $X$.
    Thus, one gets a homological marked morphism $\umfld_{/Y} \to \umfld_{/C(\pi)} \to \umfld_{/X}$.
    Note that the preimage of a basic in $Y$ is basic in $C(\pi)$, so the morphisms above send basics to basics.
    Often in this text, we will consider closed inclusions $Y \xmono[\textnormal{cl}] X$ which are only closed strata inclusions after refining the target (for example the inclusion $\mathbb{R} \xmono[\textnormal{cl}] \mathbb{R}^2$ on the first coordinate).
    The construction is identical to the one described here.
\end{enumerate}

\begin{remark}[Contractible choice of tubular neighborhood]\label{rem:tubular_neighborhood}
    Note that the choice of tubular neighborhood above is contractible.
    More precisely, the neighborhoods are given by considering the functorial ``unzipping'' construction (see \cite[Sec. $7.2$]{AFT17a}) and picking an inward vector field which is non-vanishing on the link (the boundary of the unzipping procedure).
    The space of such vector fields is contractible.
    The construction can be seen as this space of vector field mapping to $\Emb(C(\pi)_{\textnormal{unref.}}, Y)$ which itself maps to the space $\Alg{C(\pi)_{\textnormal{unref.}}}(\uDisk_{/X})_{\leq 0}$.
    Thus, any two choices of tubular neighborhoods will give rise to naturally equivalent functors $\cat{H}(\uMfld_{/X}, \V) \to \cat{H}(\uMfld_{/Y}, \V)$.
\end{remark}

The following result is the key technical tool of this paper on the factorization homology side.

\begin{lemma}[Commutative Diagrams]\label{lem:commutative_diagrams}
    Let $\V \inset \EAlg{\infty}(\Cat{1}^\geom)$.
    Any commutative diagram of homological maps between stratified spaces, which induces a commutative diagram of homological morphisms between the corresponding $\umfld_{/-}$, will also induce a commutative diagram of functors between the corresponding $\Alg{-}(\V) \simeq \cat{H}(\uMfld_{/-}, \V)$.
\end{lemma}
\begin{proof}
    As homologicality is closed under composition (\autoref{lem:homological_composition}), the total composition is homological as well.
    \autoref{lem:lifting_homology} then implies that there is a unique lift at the level of homology theories, and hence categories of algebras.
    The uniqueness of the lift precisely gives the commutativity of the diagram.
\end{proof}

\begin{remark}[An example of commutativity]\label{rem:vector_fields}
    For most diagrams of homological maps, the commutativity at the level of homological morphisms is clear from the construction.
    The only subtle case we will encounter here is for diagrams of the form
\[\begin{tikzcd}
	\bullet & \bullet \\
	\bullet & \bullet
	\arrow[two heads, from=1-1, to=1-2]
	\arrow["{\textnormal{cl}}", hook, from=1-1, to=2-1]
	\arrow["{\textnormal{cl}}", hook, from=1-2, to=2-2]
	\arrow[two heads, from=2-1, to=2-2]
\end{tikzcd}\]
    Using \autoref{rem:tubular_neighborhood}, we can pick a vector field for the right closed inclusion and pull it back to the left one.
    This induces a commutative diagram at the level of \autoref{eq:closed_inclusion}, which in turn guarantees that the conditions of \autoref{lem:commutative_diagrams} are satisfied.
\end{remark}

\subsection{Handlebodies as Duality Datum}\label{subsec:handlebodies}

The goal of this subsection is to construct the stratified spaces which allow us to describe the ``duality datum'' of a $k$-morphism.
We fix $n\geq 1$ throughout this section.
We consider various stratifications of $\mathbb{R}^n$ which will ``carry'' the algebraic structure of the duality data and homological maps (see \autoref{def:homological_morphism}) between them which will transfer these structures.
We denote the coordinates of a point $x \inset \mathbb{R}^n$ by $(x_1, \ldots, x_n)$.
The $i^\textnormal{th}$-coordinate should be viewed as the direction of composition of $i$-morphisms in an $(\infty, n)$-category.

The first stratification is to be thought of as the stratification which "carries" the structure of an arbitrary $k$-morphism:

\begin{definition}[Bar-Stratification]
    For $0 \leq k \leq n$, we define $\mathbb{R}^{n}_{\overline{k}}$ to be the stratification of $\mathbb{R}^{n}$ given by the following strata:
    \begin{enumerate}
        \item for $1 \leq i \leq k$, the strata given by $x_{1} = x_{2} = \ldots = x_{i-1} = 0$, and $\pm x_i > 0$,
        \item the stratum given by $x_{1} = x_{2} = \ldots = x_{k-1} = x_k = 0$.
    \end{enumerate}
    This is a total of $2k + 1$ strata (by convention, when $k=0$, there is a single stratum: $\mathbb{R}^n$ itself).
\end{definition}

\begin{remark}[Algebras over Bar-Stratifications]
    If $k < n$, the data of an algebra $M \inset \Alg{\mathbb{R}^{n}_{\overline{k}}}(\V)$ is precisely the data of a $k$-morphism in the $(\infty,n+1)$-category $\Mor_n(\V)$.
    If $k=n$, any such algebra gives an $n$-morphism in $\Mor_n(\V)$, together with a choice of a point (an $\mathcal{E}_0$-structure).
\end{remark}

The (co)domain and composition constructions can be seen as the following operations.
For the rest of the exposition, we fix:
\begin{enumerate}
    \item A smooth function $\rho \in \mathbb{R} \to [0, \infty)$ which is strictly decreasing on $(-\infty, 0]$, such that $\forall n \geq 0, \rho^{(n)}(0) = 0$ and $\forall x \geq 0, \rho(x) = 0$.
        For example, one may use $\rho(x) = e^{1/x}$ for all $x < 0$ and $\rho(x) = 0$ for $x \geq 0$.

    \item An increasing diffeomorphism $\phi \in \mathbb{R} \to \mathbb{R}_{>0}$.
\end{enumerate} 
The choice of these functions is contractible and does not affect the results; we omit the precise statement as it is not necessary for our purposes.

\begin{definition}[(Co)Domain]\label{def:(co)domain}
    For $1 \leq k \leq n$, we have an open embedding $\cod \in \mathbb{R}^{n}_{\overline{k-1}} \mono \mathbb{R}^{n}_{\overline{k}}$ which sends
    \[
        x_k \mapsto \phi(x_k)
    \]
    This induces an equivalence $\mathbb{R}^{n}_{\overline{k-1}} \simeq \left(\mathbb{R}^{n}_{\overline{k}}\right)_{x_k > 0}$.
    Similarly, one has $\dom \in \mathbb{R}^{n}_{\overline{k-1}} \mono \mathbb{R}^{n}_{\overline{k}}$ given by $x_k \mapsto -\phi(-x_k)$
\end{definition}

\begin{definition}[Composition]\label{def:composition}
    For $1 \leq k \leq n$, we define $\mathbb{R}^{n}_{\overline{k} \circ \overline{k}}$ to be the stratification of $\mathbb{R}^{n}$ given by the following strata:
    \begin{enumerate}
        \item for $1 \leq i < k$, the strata given by $x_{1} = x_{2} = \ldots = x_{i-1} = 0$, and $\pm x_i > 0$,
        \item the strata given by $x_{1} = x_{2} = \ldots = x_{k-1} = 0$ and $x_k < -1$, $x_k = -1$, $-1 < x_k < 1$, $x_k = 1$, $1 < x_k$.
    \end{enumerate}
    This is a total of $2k + 3$ strata.
    We have a constructible bundle $\mathbb{R}^{n}_{\overline{k} \circ \overline{k}} \xepi[\circ] \mathbb{R}^{n}_{\overline{k}}$ which sends
    \[
        x_k \mapsto x_k \left( \rho(1 - x_k^2) + \sum_{i=1}^{k-1} x_i^2 \right)
    \]
    intuitively pinching the two $(n-k)$-dimensional strata.
\end{definition}

\begin{warning}[Order of Composition]\label{warning:order}
    Note that the composition in $\Mor_1(\V)$ of $A \xto[M] B \xto[N] C$ is given by $N \circ M = M \otimes_B N$.
    The conventions are reversed: categorical composition is the reverse ordering compared to tensoring and factorization homology.
\end{warning}

The second stratification of interest is to be thought of as the stratification which ``carries'' the structure of some specific duality datum of an arbitrary $k$-morphism.

\begin{definition}[Duality Data]\label{def:duality_data}
    Let $d < n$ and a word $\sigma \inset \{\epsilon, \eta\}^d$ of length $d$ in the letters ``$\epsilon$, $\eta$'' is called a $d$-duality datum.
    We denote $|\eta| = 1$ and $|\epsilon| = -1$.
    To lighten notations, we let $\{\sigma = \eta\} = \{|\sigma |= 1\} = \{i \mid \sigma_i = \eta\}$ and $\{\sigma = \epsilon\} = \{|\sigma |= -1\} = \{i \mid \sigma_i = \epsilon\}$.
\end{definition}

\begin{definition}[Handlebody-Stratification]
    Let $1 + d \leq k + d \leq n$ and $\sigma = \sigma_{k+d-1} \ldots \sigma_{k} \inset \{\epsilon, \eta\}^d$ be a $d$-duality datum.
    We define $\mathbb{R}^{n}_{\overline{k},\sigma}$ to be the stratification of $\mathbb{R}^{n}$ given by the following strata:
    \begin{enumerate}
        \item \textbf{Bar-strata:} for $1 \leq i \leq k - 1$, the strata given by $x_{1} = x_{2} = \ldots = x_{i-1} = 0$, and $\pm x_i > 0$,
        \item \textbf{Handlebody-strata:} $x_{1} = x_{2} = \ldots = x_{k-1} = 0$, and 
        \begin{enumerate}
            \item $x_{k+d} > \sum_{i=k}^{k+d-1} |\sigma_i| x_i^2$,
            \item $x_{k+d} = \sum_{i=k}^{k+d-1} |\sigma_i| x_i^2$,
            \item $x_{k+d} < \sum_{i=k}^{k+d-1} |\sigma_i| x_i^2$.
        \end{enumerate}
    \end{enumerate}
    This is a total of $2k + 1$ strata.
\end{definition}

Note that the stratification above is isomorphic to the $k$-bar-stratification by sending $x_{k+d}$ to $x_{k+d} - \sum_{i=k}^{k+d-1} |\sigma_i| x_i^2$, and by permuting the order of the coordinates.
In particular, algebras over $\mathbb{R}^{n}_{\overline{k},\sigma}$ are also (pointed) $k$-morphisms in $\Mor_n(\V)$, although maybe in a less canonical way (as it depends on the choice of permutation).

Finally, the third stratification is to be thought of as the stratification which ``carries'' the understanding that the $\sigma$-duality datum of a $k$-morphism forms a $(k+d)$-morphism.
We encourage the reader to skip this definition on a first reading, and to come back to it after \autoref{def:full_pinch}.
Indeed, the stratification might seem quite complicated, but it is simply the minimal one which makes the ``full-pinch'' map a stratified one.

\begin{definition}[Fully-Barred Handlebody-Stratification]\label{def:bar_handlebody_strata}
    Let $1 + d \leq k + d \leq n$ and $\sigma = \sigma_{k+d-1} \ldots \sigma_{k} \inset \{\epsilon, \eta\}^d$ be a $d$-duality datum.
    We define $\mathbb{R}^{n}_{\overline{k},\overline{\sigma}}$ to be the stratification of $\mathbb{R}^{n}$ given by the following strata:
    \begin{enumerate}
        \item \textbf{Bar-strata:} for $1 \leq i \leq k - 1$, the strata given by $x_{1} = x_{2} = \ldots = x_{i-1} = 0$, and $\pm x_i > 0$,
        \item \textbf{Handlebody-Bar-strata:}
        For $k \leq i < k+d$, write $\Delta_i = |\sigma_i| \left(x_{k+d} - \sum_{j=i}^{k+d-1} |\sigma_i| x_i^2\right)$ with the convention $\Delta_{k+d} = -1$.
        For $k \leq j \leq k+d$, we have the following strata
        \begin{enumerate}
            \item the $2$ strata given by $\pm x_{j} > 0$, $\Delta_j < 0$ and $x_u = 0$ for $u \inset \{1, \ldots, j-1\}$,
            \item for $k \leq i < j$, the $3$ strata given by $x_u = 0$ for $u \inset \{1, \ldots, k-1\} \sqcup \{|\sigma| = |\sigma_i|\}\cap\{k, \ldots, i-1\} \sqcup \{|\sigma| = -|\sigma_i|\}\cap\{k, \ldots, j-1\}$ and by either
            \begin{enumerate}
                \item $\Delta_i > 0$, or
                \item $\Delta_i = 0$ and $\pm x_i > 0$. 
            \end{enumerate}
        \end{enumerate}
        Finally, we have the stratum given by $x_1= \ldots = x_{k+d}=0$.
    \end{enumerate}
\end{definition}

Note that we have an obvious refinement map $\mathbb{R}^{n}_{\overline{k},\overline{\sigma}} \dashrightarrow \mathbb{R}^{n}_{\overline{k},\sigma}$.

We now define some "pinch maps" which will help formalize the idea that the $\sigma$-duality datum of an arbitrary $k$-morphism is a $(k+d)$-morphism.

\begin{definition}[Last Coordinate Pinch]
    As above, let $1 + d \leq k + d \leq n$ and $\sigma = \sigma_{k+d-1} \ldots \sigma_{k}\inset \{\epsilon, \eta\}^d$ be a $d$-duality datum.
    We define a stratified map $\pi_k \in \mathbb{R}^{n}_{\overline{k},\overline{\sigma}} \to \mathbb{R}^{n}_{\overline{k+1},\overline{\sigma_{k+d-1} \ldots \sigma_{k+1}}}$ called the ``pinch map'' by the following:
    \begin{enumerate}
        \item it leaves unchanged all coordinates except the $k^{\textnormal{th}}$-one,
        \item it sends the $k^{\textnormal{th}}$-coordinate $x_{k}$ to
        \[
            x_{k} \longmapsto x_{k} \left( \rho\left(|\sigma_k|\left(x_{k+d} - \sum_{i=k}^{k+d-1} |\sigma_i| x_i^2\right)\right) + \sum_{i=1}^{k-1} x_i^2\right)
        \]
    \end{enumerate}
\end{definition}

\begin{figure}[H]
\centering
\begin{tikzpicture}
    \def\scaleFactor{0.5}  
    \def\imageSeparation{10}  
    
    \node (img1) at (0, 0) {\includegraphics[scale=\scaleFactor]{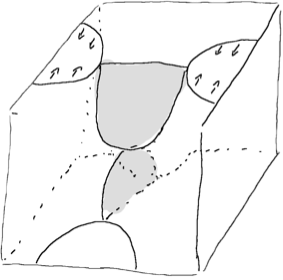}};
    
    \node (img2) at (\imageSeparation, 0) {\includegraphics[scale=\scaleFactor]{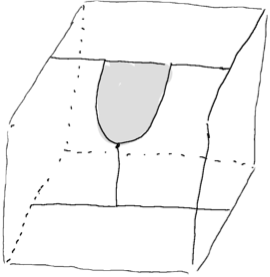}};
    
    \draw[->>, thick] (img1) -- (img2) node[midway, above] {$\pi_1$};
\end{tikzpicture}
\caption{
    Pinching $\mathbb{R}^{3}_{\overline{1},\overline{\epsilon,\eta}} \xepi[\pi_1] \mathbb{R}^{3}_{\overline{1},\overline{\eta}}$.
    The small arrows indicate the collapsing that occurs ``inside'' the handle.
}
\end{figure}

While we only need the easier notion of weakly constructible bundle for our purposes, one can easily see that these maps are constructible bundles.

\begin{lemma}
    As above, let $1 + d \leq k + d \leq n$ and $\sigma = \sigma_{k+d-1} \ldots \sigma_{k} \inset \{\epsilon, \eta\}^d$ be a $d$-duality datum.
    The pinch
    \[
        \mathbb{R}^{n}_{\overline{k},\overline{\sigma}} \xepi[\pi_k] \mathbb{R}^{n}_{\overline{k+1},\overline{\sigma_{k+d-1} \ldots \sigma_{k+1}}}
    \]
    is a weakly constructible bundle, \textit{i.e.} a stratified fiber bundle when corestricted to each stratum of the codomain.
\end{lemma}
\begin{proof}
    Consider a stratum of the codomain $\mathbb{R}^{n}_{\overline{k+1},\overline{\sigma_{k+d-1} \ldots \sigma_{k+1}}}$.
    We have to check that the corestriction of pinch map on this stratum is a stratified fiber bundle over it.
    To simplify notations, we write $\overline{x}_k$ for the new $k^{\textnormal{th}}$-coordinate after applying the pinch map.
    We denote the strata using a ``path'' notation indicating the order of the choices.
    \begin{enumerate}
        \item \textbf{Bar-strata:} The bar-strata given by $i.\pm$ are unchanged for $i \leq k-1$.
        The preimages of the bar-strata $k.\pm$ are given by the handlebody-bar-strata $k.a.\pm$.
        For both these cases, the pinch map is diffeomorphic.

        \item \textbf{Handlebody-Bar-Strata:} Let $k + 1 \leq j \leq k+d$. If $\sigma_k = \sigma_j$, the preimage of the handlebody-bar-strata $j.a.\pm$ are given by the handlebody-bar-strata $j.a.\pm$, and the pinch is diffeomorphic.
        If $\sigma_k \neq \sigma_j$, the preimages of the handlebody-bar-strata $j.a.\pm$ are given by the union of the three handlebody-bar-strata $j.k.\_$, where the pinch is pointwise given by a map equivalent to $[-1,1] \to \{0\}$.
        On these strata, one can check that the pinch is a fiber bundle.
        For $k+1 \leq i \leq j$, the preimages of the handlebody-bar-strata $j.i.\_$ are given by the corresponding handlebody-bar-strata $j.i.\_$, where the pinch is diffeomorphic.

        \item \textbf{Bottom-Bar-Stratum:} Corestricted to the stratum given by $x_1 = \ldots = x_{k+d} = 0$, the pinch map is the identity, and hence a fiber bundle.
    \end{enumerate}
\end{proof}

In particular, we can compose these pinch maps to get the corollary below.

\begin{definition}[Full Pinch]\label{def:full_pinch}
    As above, let $1 + d \leq k + d \leq n$ and $\sigma = \sigma_{k+d-1} \ldots \sigma_{k} \inset \{\epsilon, \eta\}^d$ be a $d$-duality datum.
    We define the full pinch as the compositions
    \[
        \mathbb{R}^{n}_{\overline{k},\overline{\sigma}} \epi \mathbb{R}^{n}_{\overline{k+1},\overline{\sigma_{k+d-1} \ldots \sigma_{k+1}}} \epi \ldots \epi \mathbb{R}^{n}_{\overline{k+d-1},\overline{\sigma_{k+d-1}}} \epi \mathbb{R}^{n}_{\overline{k+d},\overline{\emptyset}} = \mathbb{R}^{n}_{\overline{k+d}}
    \]
    We define $\sigma \in \Alg{\mathbb{R}_{\overline{k}}^n}(\V) \to \Alg{\mathbb{R}_{\overline{k+d}}^n}(\V) $ to be the composition of the functors induced by the homological maps (see \autoref{subsec:homological})
    \[
        \mathbb{R}_{\overline{k}}^n \longeq \mathbb{R}^{n}_{\overline{k},\sigma} \dashleftarrow \mathbb{R}^{n}_{\overline{k},\overline{\sigma}} \longepi \mathbb{R}^{n}_{\overline{k+d}}
    \]
\end{definition}

\begin{corollary}
    As above, let $1 + d \leq k + d \leq n$ and $\sigma \inset \{\epsilon, \eta\}^d$ be a $d$-duality datum.
    The full pinch $\mathbb{R}^{n}_{\overline{k},\overline{\sigma}} \to \mathbb{R}^{n}_{\overline{k+d}}$ is a weakly constructible bundle, \textit{i.e.} a fiber bundle when restricted to each stratum of the codomain.
\end{corollary}
\begin{proof}
    This is a direct consequence of \cite[Prop. $6.1.8$]{AFR18a}.
\end{proof}

We briefly check that, for $d = 1$, the notation $\eta(f)$ and $\epsilon(f)$ introduced above do indeed exhibit an adjunction in the higher Morita category.

\begin{lemma}[Duality Data Form Adjunctions]\label{lem:graphical_adjunction}
    Let $k < n$ and $f \inset \Alg{\mathbb{R}_{\overline{k}}^n}(\V)$ be a $k$-morphism in $\Mor_n(\V)$.
    Let $f^{\circ}$ be the $k$-morphism induced by $x_k \mapsto -x_k$ on $\mathbb{R}^{n}_{\overline{k}}$.
    Then, the $(k+1)$-morphisms $\eta(f)$ and $\epsilon(f)$ are respectively the unit and counit of an adjunction $f \dashv f^{\circ}$ in $\Mor_n(\V)$.
\end{lemma}
\begin{proof}
    By inspection, $\cod(\eta(f)) \simeq f^{\circ} \circ f$ and $\dom(\epsilon(f)) \simeq f \circ f^{\circ}$ using \autoref{def:(co)domain} and \autoref{def:composition}.
    The other domain and codomain can similarly be identified.
    The triangle identities then follow from the usual graphical calculus for adjunctions.
\end{proof}

Note that, by construction, we have the exchange property $\left(f^{\circ}\right)^{\circ} \simeq f$ and $\eta(f^{\circ}) \simeq \epsilon(f)^{\circ}$.
We then have a direct corollary of \autoref{prop:free_full_adj} and \autoref{lem:graphical_adjunction} above.

\begin{corollary}[Units in Morita]\label{cor:units_in_morita}
    Let $k < n$ and $f \inset \Alg{\mathbb{R}_{\overline{k}}^n}(\V)$ be a $k$-morphism in $\Mor_n(\V)$.
    We have a lift $c_k \epi \Free^{\{k\}}(c_k) \to \Mor_n(\V)$ where all the induced (co)units are given by either of the following four algebras: $\eta(f)$, $\epsilon(f)$, $\eta(f^{\circ})$, $\epsilon(f^{\circ})$.
\end{corollary}

We now prove an associativity lemma that can be understood as ``$\sigma(\eta(M)) \simeq \sigma\eta(M)$'' (resp. for $\epsilon$).
This lemma is key to our inductive process in the proof of \autoref{thm:full_dualizability_criterion}.

\begin{lemma}[Left Associativity]\label{lem:associativity}
    Let $1 + d \leq k + d \leq n$ with $d \geq 2$ and let $\sigma = \sigma_{k+d-1} \ldots \sigma_{k} \inset \{\epsilon, \eta\}^d$ be a $d$-duality datum.
    Define $\sigma' = \sigma_{k+d-1} \ldots \sigma_{k+2} \inset \{\epsilon, \eta\}^{d-1}$.
    We have a commutative diagram
\[\begin{tikzcd}
	& {\Alg{\mathbb{R}^n_{\overline{k+1}}}(\V)} \\
	{\Alg{\mathbb{R}^n_{\overline{k}}}(\V)} && {\Alg{\mathbb{R}^n_{\overline{k+d}}}(\V)}
	\arrow["{\sigma'}", from=1-2, to=2-3]
	\arrow["{\sigma_{k+1}}", from=2-1, to=1-2]
	\arrow["\sigma"', from=2-1, to=2-3]
\end{tikzcd}\]
\end{lemma}
\begin{proof}
    One can define a ``partially-bar'' handlebody stratification $\mathbb{R}^{n}_{\overline{k},\overline{\sigma_k},\sigma'}$, refining $\mathbb{R}^{n}_{\overline{k},\sigma}$ and minimal such that the pinch map $\mathbb{R}^{n}_{\overline{k},\overline{\sigma_k},\sigma'} \xto[\pi_k] \mathbb{R}^{n}_{\overline{k+1},\sigma'}$ is a map of stratified spaces.
    We then have a commutative diagram of homological maps between stratified spaces
\[\begin{tikzcd}
	{\mathbb{R}^n_{\overline{k}}} & {\mathbb{R}^n_{\overline{k}, \sigma_{k+1}}} & {\mathbb{R}^n_{\overline{k}, \overline{\sigma_{k+1}}}} & {\mathbb{R}^n_{\overline{k+1}}} \\
	{\mathbb{R}^n_{\overline{k}, \sigma}} && {\mathbb{R}^n_{\overline{k}, \overline{\sigma_{k+1}}, \sigma'}} & {\mathbb{R}^n_{\overline{k+1}, \sigma'}} \\
	{\mathbb{R}^n_{\overline{k}, \overline{\sigma}}} &&& {\mathbb{R}^n_{\overline{k+1}, \overline{\sigma'}}} & {\mathbb{R}^n_{\overline{k+d}}}
	\arrow["\sim"{marking, allow upside down}, no head, from=1-1, to=1-2]
	\arrow["\sim"{marking, allow upside down}, no head, from=1-1, to=2-1]
	\arrow[dashed, from=1-3, to=1-2]
	\arrow[two heads, from=1-3, to=1-4]
	\arrow["\sim"{marking, allow upside down}, no head, from=1-3, to=2-3]
	\arrow["\sim"{marking, allow upside down}, no head, from=1-4, to=2-4]
	\arrow[dashed, from=2-3, to=2-1]
	\arrow[two heads, from=2-3, to=2-4]
	\arrow[dashed, from=3-1, to=2-1]
	\arrow[dashed, from=3-1, to=2-3]
	\arrow[two heads, from=3-1, to=3-4]
	\arrow[dashed, from=3-4, to=2-4]
	\arrow[two heads, from=3-4, to=3-5]
\end{tikzcd}\]
    Clearly, this diagram satisfies the condition of \autoref{lem:commutative_diagrams} and hence induces a commutative diagram of functors between the corresponding categories of algebras.
    The outer top composition induces $\sigma' \circ \sigma_{k+1}$, while the outer bottom composition induces $\sigma$.
\end{proof}

\subsection{Dualizability of Iterated Bimodules}\label{subsec:dual}

This subsection is solely devoted to the result below.
We adapt here an argument from \cite[Prop. $4.6.2.13$]{Lur17} (which is the $n = 1$ case of our result below), combining it with the idea from \cite[Cor. $3.8$]{Ram24a} (which gives conditions for when an (op)lax-module functor might be an module functor).

\begin{construction}[Actions for \autoref{prop:existence_dual}]\label{constr:actions_existence_dual}
    Let $M$ be a $\mathbb{R}^n_{\overline{n}}$-algebra in $\V$.
    We have a homological map given by the closed strata inclusion $\mathbb{R}_{\leq 0} \xmono[\textnormal{cl}] \mathbb{R}^n_{\overline{n}}$ of the last coordinate (resp. for $\geq$).
    By \autoref{lem:lifting_homology}, this induces a $\mathbb{R}_{\leq 0}$-algebra (resp. $\mathbb{R}_{\geq 0}$-algebra) that we denote $A \curvearrowright M$ (resp. $M \curvearrowleft B$).
    This algebra can be seen as $1$-morphisms $M \in A \to \mathbbm{1}_{\V} \in \Mor_1(\V)$ (resp. $M \in \mathbbm{1}_{\V} \to B \in \Mor_1(\V)$).
    We call $A$ (resp. $B$) the domain (resp. codomain) of $M$.
    Note that this definition corresponds to using \autoref{lem:lifting_homology} on \autoref{def:(co)domain}, and additionally forgetting the extra structure of a $\mathbb{R}^{n}_{\overline{n-1}}$-algebra by pushing forward through the fiber bundle $\mathbb{R}^{n}_{\overline{n-1}} \epi \mathbb{R}$ projecting on the last coordinate.
\end{construction}

\begin{proposition}[Dualizability for Iterated Bimodules]\label{prop:existence_dual}
    Let $M$ be a $\mathbb{R}^n_{\overline{n}}$-algebra in $\V$, seen as an $n$-morphism in $\Mor_n(\V)$.
    Using the notations of \autoref{constr:canonical_action_units} above,
    \begin{enumerate}
        \item $M$ is a left adjoint in $\Mor_n(\V)$ if and only if the induced $\mathbb{R}_{\geq 0}$-algebra $M \curvearrowleft B$, seen as a $1$-morphism $M \in \mathbbm{1}_{\V} \to B \in \Mor_1(\V)$, is a left adjoint,
        \item $M$ is a right adjoint in $\Mor_n(\V)$ if and only if the induced $\mathbb{R}_{\leq 0}$-algebra $A \curvearrowright M$, seen as a $1$-morphism $A \in \Mor_1(\V) \to M \in \mathbbm{1}_{\V}$, is a right adjoint.
    \end{enumerate}
\end{proposition}

\begin{warning}[Adjointable]\label{warning:adjointable}
    Note that Lurie uses the term ``left adjointable'' to mean what we call here ``being a right adjoint'' or ``having a left adjoint''.
\end{warning}

\begin{proof}
    The second statement follows from the first by the ``$n$-op'' operation $x_{n} \mapsto - x_{n}$.
    Consider the following composition of constructible bundles $\mathbb{R}^n_{\overline{n}} \epi \left(\mathbb{R}^2_{\overline{2}}\right)_{r \geq 0} \epi \mathbb{R}^1_{\overline{1}}$, where the coordinates of the middle copy of $\mathbb{R}^2$ are denoted $(x_n, r)$, defined by
    \[
        (x_1, \ldots, x_n) \longmapsto \left(x_n, \sqrt{\sum_{i=1}^{n-1} x_i^2} \right) \longmapsto x_n
    \]
    Using the iterated construction of the Morita category \cite[Cor. $7.6$]{Hau23}, or the factorization homology construction of composition in $\Mor_n(\V)$, we see that $M$ is a left adjoint in $\Mor_n(\V)$ if and only if it is a left adjoint as a $1$-morphism in $\Mor_1(\Mod_W(\V))$ where $W$ is the $\mathcal{E}_2$-algebra supported on $r > 0$ in the pushforward $(x_n, r)^*(M)$.
    By \cite[Prop. $4.6.2.13$]{Lur17}, this is equivalent to the induced $\mathbb{R}_{\geq 0}$-algebra being a left adjoint in $\Mor_1(\Mod_W(\V))$.
    It remains to show that this is equivalent to being a left adjoint in $\Mor_1(\V)$.
    One direction comes from stability of adjunctions under composition, so it suffices to show that being a left adjoint in $\Mor_1(\V)$ implies being a left adjoint in $\Mor_1(\Mod_W(\V))$.
    The first part of the argument relies on a similar idea to the proof of \cite[Prop. $4.6.2.13$]{Lur17}, while the second part is a consequence of (a harmless generalization) of the argument in \cite[Cor. $3.8$]{Ram24a}.

    Let $\cat{W} = \Mod_W(\V)$ and let $B$ denote the induced $\mathbb{R}_{\geq 0}$-algebra.
    Note that $\Mod_B(\V) \simeq \Mod_B(\cat{W})$, so we can abuse notation and denote both categories by $\Mod_B$.
    Consider the following lifting problem:
\begin{equation}\label{eq:W_lift}
\begin{tikzcd}
	&& {\cat{W}} \\
	\\
	\V &&&& {\Mod_B}
	\arrow[""{name=0, anchor=center, inner sep=0}, curve={height=6pt}, from=1-3, to=3-1]
	\arrow[""{name=1, anchor=center, inner sep=0}, "{W \otimes -}"', curve={height=6pt}, from=3-1, to=1-3]
	\arrow[""{name=2, anchor=center, inner sep=0}, "{B^{\vee}\otimes -}", curve={height=-6pt}, from=3-1, to=3-5]
	\arrow[""{name=3, anchor=center, inner sep=0}, "{B \otimes -}", curve={height=-6pt}, from=3-5, to=1-3]
	\arrow[""{name=4, anchor=center, inner sep=0}, curve={height=6pt}, dashed, from=3-5, to=1-3]
	\arrow[""{name=5, anchor=center, inner sep=0}, "{B \otimes -}", curve={height=-6pt}, from=3-5, to=3-1]
	\arrow["\dashv"{anchor=center, rotate=141}, draw=none, from=1, to=0]
	\arrow["\dashv"{anchor=center, rotate=-90}, draw=none, from=2, to=5]
	\arrow["\dashv"{anchor=center, rotate=-148}, draw=none, from=4, to=3]
\end{tikzcd}
\end{equation}
    where we'd like to show the dotted left adjoint exists.
    We can consider the full subcategory $\cat{W}' \mono \cat{W}$ on the objects $w$ such that $\cat{W}(w, B \otimes -) \in \Mod_B \to \Spa$ is corepresentable.
    Clearly, $\cat{W}'$ contains the free $W$-algebras (the image of the functor $\V \to \cat{W}$) and is closed under geometric realizations.
    Hence, $\cat{W}' = \cat{W}$ and the left adjoint exists.

    Automatically, this left adjoint preserves colimits.
    To show that it is given by tensoring with some lift $B^{\vee} \inset \cat{W}$, it suffices to check that this adjoint is a $\cat{W}$-module map.
    This follows from the argument in \cite[Cor. $3.8$]{Ram24a}.
    As a left adjoint to a module map, it is automatically an oplax-$\cat{W}$ map.
    Because $\cat{W}$ is generated under colimits by its unit $W$, it is in particular generated under colimits by dualizable objects.
    The ``trace map'' construction of \cite[Prop. $3.4$]{Ram24a} then shows that the oplax structure maps are equivalences.
\end{proof}

In particular, we immediately get the following corollary in the category of spectra.

\begin{corollary}[Dualizability for Iterated Bimodules in Spectra]\label{cor:existence_dual_spectra}
    Let $M$ be a $\mathbb{R}^n_{\overline{n}}$-algebra in $\Spe$, seen as an $n$-morphism in $\Mor_n(\Spe)$.
    Using the notations of \autoref{constr:canonical_action_units} above,
    \begin{enumerate}
        \item $M$ is a left adjoint in $\Mor_n(\V)$ if and only if it is compact as a module over the induced $\mathbb{R}_{> 0}$-algebra $B$, if and only if it is perfect as a module over the induced $\mathbb{R}_{> 0}$-algebra $B$,
        \item $M$ is a right adjoint in $\Mor_n(\V)$ if and only if it is compact as a module over the induced $\mathbb{R}_{< 0}$-algebra $A$, if and only if it is perfect as a module over the induced $\mathbb{R}_{< 0}$-algebra $A$.
    \end{enumerate}
\end{corollary}
\begin{proof}
    This is a direct consequence of \cite[$7.2.4.4$]{Lur17} which characterizes dualizable modules in $\Spe$.
\end{proof}

\subsection{Centers for Dualizable Algebras}\label{subsec:deligne}

We recall here the content of \cite[Sec. $5.3.1$]{Lur17} and detail how it applies to our setting.

\begin{lemma}[Units as Maps to Centers]\label{lem:unit_center}
    Let $(\cat{A},\cat{M}) \inset \Alg{\mathbb{R}_{\geq 0}}\left(\Cat{1}^{\geom}\right)$ and $(A,M)$ a $\uDisk_{/\mathbb{R}_{\geq 0}}$-algebra in $(\cat{A},\cat{M})$.
    Suppose that we can extend the diagram below with the following dashed right adjoints.
\begin{equation}\label{eq:center_lift}\begin{tikzcd}
	&& {\Mod_A(\cat{A})} \\
	\\
	\\
	{\cat{A}} &&&& {\cat{M}}
	\arrow[""{name=0, anchor=center, inner sep=0}, curve={height=-6pt}, from=1-3, to=4-1]
	\arrow[""{name=1, anchor=center, inner sep=0}, "{-\otimes_A M}", curve={height=-6pt}, from=1-3, to=4-5]
	\arrow[""{name=2, anchor=center, inner sep=0}, "{-\otimes A}", curve={height=-6pt}, from=4-1, to=1-3]
	\arrow[""{name=3, anchor=center, inner sep=0}, "{-\otimes M}", curve={height=-6pt}, from=4-1, to=4-5]
	\arrow[""{name=4, anchor=center, inner sep=0}, "{\hom_{\cat{M}}(M,-)}", curve={height=-6pt}, dashed, from=4-5, to=1-3]
	\arrow[""{name=5, anchor=center, inner sep=0}, "{\hom_{\cat{M}}(M,-)}", curve={height=-6pt}, dashed, from=4-5, to=4-1]
	\arrow["\dashv"{anchor=center, rotate=-154}, draw=none, from=1, to=4]
	\arrow["\dashv"{anchor=center, rotate=-90}, draw=none, from=3, to=5]
	\arrow["\dashv"{anchor=center, rotate=-27}, draw=none, from=2, to=0]
\end{tikzcd}\end{equation}
    Then, when evaluating at $\mathbbm{1}_{\cat{A}}$, the monad of the bottom adjunction gives the center of $M$ (seen as an object in $\cat{A}$) and the induced map of monads gives the universal $A \to \hom_{\cat{M}}(M,M)$ corresponding to the action of $A \curvearrowright M$ (seen as a morphism in $\cat{A}$).
\end{lemma}

Note that we state and use this claim only regarding the map $A \to \hom_{\cat{M}}(M,M)$ as a map in $\cat{A}$.
It is also true as a map in $\EAlg{1}(\cat{A})$, using the induced $\mathcal{E}_1$-algebra structure on the map of monads, but we will not need this fact here.

\begin{proof}
    By definition, $\hom(M,M)$ is an endomorphism object for $M$ \cite[Def. $4.2.1.28$]{Lur17}, and hence a center for $M$ \cite[Lem. $5.3.1.11$]{Lur17}.
    By the universal property of morphism objects, to identify maps $A \to \hom_{\cat{M}}(M,M)$, it suffices to identify the induced maps $A \otimes M \to M$.
    It is then given by construction that the map of monads induces the correct action map on $M$.
\end{proof}

\begin{remark}[Lift of the Action of $A$ on the Dual]
    One can actually prove the top dashed arrow exists if and only if the bottom dashed arrow exists.
    This is essentially contained in \cite[Sec. $5.3.1$]{Lur17}.
    One way to see this is to observe that the left hand side adjunction is monadic and following the ideas of \cite[Sec. $4.7.3$]{Lur17}, we have:
    a (left) action of $A$ on $M$ gives a (right) action of the monad $- \otimes A$ on $- \otimes M$,
    which in turn is equivalent to a (left) action of the monad $- \otimes A$ on $\hom_{\cat{M}}(M,-)$,
    which in turn is equivalent to the existence of the dotted lift.
\end{remark}

To complete the picture, we propose the following approach to re-prove the results of \cite[Sec. $5.3.1$]{Lur17}.
Now suppose $(\cat{A},\cat{M}) \inset \Alg{\mathbb{R}_{\leq 0} \times \mathbb{R}^n}(\Cat{1}^{\geom})$ and $(A,M)$ a $\uDisk_{/\mathbb{R}_{\leq 0} \times \mathbb{R}^n}$-algebra in $(\cat{A},\cat{M})$.
Consider a similar diagram, but replacing $\cat{M}$ with $\Mod^{\multisided}_M(\cat{M})$, and assume again the existence of the dashed right adjoints.
All the left adjoint functors are $\mathcal{E}_n$-monoidal.
Then, as can also be seen by \cite[Cor. $C$]{HHLN24}, all the right adjoint functors are canonically lax-$\mathcal{E}_n$-monoidal and the diagram induces
\[\begin{tikzcd}
	&& {\EAlg{n}(\Mod_A(\cat{A}))} \\
	\\
	\\
	{\EAlg{n}(\cat{A})} &&&& {\EAlg{n}(\Mod^{\multisided}_M(\cat{M}))} & {\EAlg{n}(\cat{M})_{M \backslash}}
	\arrow[""{name=0, anchor=center, inner sep=0}, curve={height=-6pt}, from=1-3, to=4-1]
	\arrow[""{name=1, anchor=center, inner sep=0}, "{-\otimes_A M}", curve={height=-6pt}, from=1-3, to=4-5]
	\arrow[""{name=2, anchor=center, inner sep=0}, "{-\otimes A}", curve={height=-6pt}, from=4-1, to=1-3]
	\arrow[""{name=3, anchor=center, inner sep=0}, "{-\otimes M}", curve={height=-6pt}, from=4-1, to=4-5]
	\arrow[""{name=4, anchor=center, inner sep=0}, "{\hom_{M}^{\multisided}(M,-)}", curve={height=-6pt}, dashed, from=4-5, to=1-3]
	\arrow[""{name=5, anchor=center, inner sep=0}, "{\hom_{M}^{\multisided}(M,-)}", curve={height=-6pt}, from=4-5, to=4-1]
	\arrow["\sim", no head, from=4-5, to=4-6]
	\arrow["\dashv"{anchor=center, rotate=-149}, draw=none, from=1, to=4]
	\arrow["\dashv"{anchor=center, rotate=-90}, draw=none, from=3, to=5]
	\arrow["\dashv"{anchor=center, rotate=-30}, draw=none, from=2, to=0]
\end{tikzcd}\]
Informally, the induced monad of the bottom adjunction, evaluated at the identity $\mathbbm{1}_{\cat{A}}$, upgrades to $\EAlg{1}(\EAlg{n}(\cat{A})) \simeq \EAlg{n+1}(\cat{A})$ and can be identified with the $\mathcal{E}_n$-center of $M$.
The induced map of monad, evaluated at the identity $\mathbbm{1}_{\cat{A}}$, can be identified with the universal $\mathcal{E}_{n+1}$-map to the center $A \to \hom^{\multisided}_M(M,M)$.
Thankfully, we do not need here to access any of the $\mathcal{E}_{n+1}$-monoidal structures of these maps, so the following result suffices for our purposes.

\begin{lemma}[Units as Maps to $\mathcal{E}_n$-Centers]\label{lem:unit_En_center}
    As above, let $(\cat{A},\cat{M}) \inset \Alg{\mathbb{R}_{\leq 0} \times \mathbb{R}^n}(\Cat{1}^{\geom})$ and $(A,M)$ a $\uDisk_{/\mathbb{R}_{\leq 0} \times \mathbb{R}^n}$-algebra in $(\cat{A},\cat{M})$.
    Suppose that we can extend the diagram below with the following dashed right adjoints.
\begin{equation}\label{eq:En_center_lift}\begin{tikzcd}
	&& {\Mod_A(\cat{A})} \\
	\\
	\\
	{\cat{A}} &&&& {\Mod^{\multisided}_M(\cat{M})}
	\arrow[""{name=0, anchor=center, inner sep=0}, curve={height=-6pt}, from=1-3, to=4-1]
	\arrow[""{name=1, anchor=center, inner sep=0}, "{-\otimes_A M}", curve={height=-6pt}, from=1-3, to=4-5]
	\arrow[""{name=2, anchor=center, inner sep=0}, "{-\otimes A}", curve={height=-6pt}, from=4-1, to=1-3]
	\arrow[""{name=3, anchor=center, inner sep=0}, "{-\otimes M}", curve={height=-6pt}, from=4-1, to=4-5]
	\arrow[""{name=4, anchor=center, inner sep=0}, "{\hom_{M}^{\multisided}(M,-)}", curve={height=-6pt}, dashed, from=4-5, to=1-3]
	\arrow[""{name=5, anchor=center, inner sep=0}, "{\hom_{M}^{\multisided}(M,-)}", curve={height=-6pt}, dashed, from=4-5, to=4-1]
	\arrow["\dashv"{anchor=center, rotate=-151}, draw=none, from=1, to=4]
	\arrow["\dashv"{anchor=center, rotate=-90}, draw=none, from=3, to=5]
	\arrow["\dashv"{anchor=center, rotate=-27}, draw=none, from=2, to=0]
\end{tikzcd}\end{equation}
    Then, when evaluating at $\mathbbm{1}_{\cat{A}}$, the monad of the bottom adjunction gives the $\mathcal{E}_n$-center of $M$ (seen as an object in $\cat{A}$) and the induced map of monads gives the universal $A \to \hom_{\cat{M}}(M,M)$ corresponding to the action of $A \curvearrowright M$ (seen as a morphism in $\cat{A}$). 
\end{lemma}
\begin{proof}
    This is an immediate consequence of \autoref{lem:unit_center} and of the proof of the main Theorem of Lurie's section on centers \cite[Thm. $5.3.1.30$]{Lur17}.
\end{proof}

In our case, the $- \otimes_A M$ functor on the top right side of the diagram can be defined explicitly with factorization homology as follows.

\begin{remark}[$\cat{E}_n$-``Bimodule'' Version]\label{rem:en_bimodule_version}
    Let $(A,M,B) \inset \Alg{\mathbb{R}^{n+1}_{\overline{1}}}(\V)$.
    We can consider $(\cat{A}, \cat{M}) = (\V, \Mod_B(\V))$ above, which canonically inherits the correct algebra structure via the collapsing of the $x_1 \geq 0$ region $\mathbb{R}^{n+1}_{\overline{1}} \epi \mathbb{R}_{\leq 0} \times \mathbb{R}^{n}$.
    The functor $\Mod_A(\V) \xto[-\otimes_A M] \Mod_M^{\multisided}(\Mod_B(\V))$ above is then given by the bimodule action induced by the homological maps
\[\begin{tikzcd}
	{\mathbb{R}^{n+1}_{\overline{1}}} & {\left(\mathbb{R}_{\leq 0} \times \{0\}^n\right) \bigsqcup_{\{0\}^{n+1}} \left(\mathbb{R}_{\geq 0} \times \mathbb{R}^{n}\right)} & {\mathbb{R}_{\overline{1}}}
	\arrow["{\textnormal{cl.}}"', hook', from=1-2, to=1-1]
	\arrow[two heads, from=1-2, to=1-3]
\end{tikzcd}\]
    where the second map is the constructible bundle induced by taking the radius in the $x_{1} > 0$ region.
\end{remark}

Using this, we can now prove the main result of this subsection, which we formulate with the following constructon.

\begin{construction}[Actions in \autoref{prop:identifying_units}]\label{constr:canonical_action_units}
    Let $0 \leq k < n$ and $a, b \geq 0$ with $a + b = n-k = d$.
    Let $M$ be a $\mathbb{R}^n_{\overline{k}}$-algebra in $\V$.
    We have the homological maps
\[\begin{tikzcd}
	{\mathbb{R}^{n}_{\overline{k}}} \\
	{(\{0\}^{a+1} \times \mathbb{R}^b \subset \{0\} \times \mathbb{R}^d \subset \mathbb{R}_{\geq 0} \times \mathbb{R}^d) \bigsqcup_{\{0\}^{a+1} \times \mathbb{R}^b} (\mathbb{R}_{\leq 0} \times \{0\}^a \times \mathbb{R}^b)} & {\mathbb{R}^{b+1}_{\overline{1}}}
	\arrow["{\textnormal{cl}}", hook', from=2-1, to=1-1]
	\arrow[two heads, from=2-1, to=2-2]
\end{tikzcd}\]
    where the vertical inclusion is induced by the inclusion of the last $d+1$ coordinates $\mathbb{R}^{d+1}_{\overline{1}} \xmono[\textnormal{cl.}] \mathbb{R}^{n}_{\overline{k}}$. 
    If $k > 0$, we denote $A$ (resp. $B$) the induced $\mathbb{R}_{<0}$-algebra (resp. $\mathbb{R}_{>0}$-algebra) described in \autoref{constr:actions_existence_dual}.
    If $k = 0$, we simply denote $A = B = \mathbbm{1}_{\V}$ the unit of $\V$.
    By \autoref{lem:lifting_homology}, these homological maps induce a $\mathbb{R}^{b+1}_{\overline{1}}$-algebra $\int_{(S^{a-1}, D^a)} (M, A) \EActright{b} M \EActleft{b} B$.
    Using \autoref{rem:en_bimodule_version} and \autoref{lem:unit_En_center}, we then get a \textbf{universal $\cat{E}_{b+1}$-map}
    \[
        \int_{(S^{a-1}, D^a)} (M, A) \longto \HC{b}^B(M)
    \]
    where $\HC{b}^B(M)$ is the $\mathcal{E}_b$-center with $(\cat{A}, \cat{M}) =(\V, \Mod_B(\V))$.
\end{construction}

\begin{proposition}[Identifying Units with Universal Maps to Centers]\label{prop:identifying_units}
    Let $n,k,d,a,b,M,A,B$ be as described in \autoref{constr:canonical_action_units} above.
    Let $\sigma = \sigma_{k+d-1} \ldots \sigma_{k} \inset \{\epsilon, \eta\}^d$ be a $d$-duality datum such that $a = |\{ k \leq i < k+d \mid \sigma_i = \epsilon\}|$.
    Suppose that $\sigma(M)$, seen as a $n$-morphism in $\Mor_n(\V)$, has a right adjoint $\sigma(M)^{\vee}$.
    Then, the image in $\V$ of the unit of the adjunction can be identified with the image of $\V$ of the universal $\mathcal{E}_{b+1}$-map
    \[
        \int_{(S^{a-1}, D^a)} (M, A) \longto \HC{b}^B(M)
    \]
    described above.
\end{proposition}

Note that the unit of the adjunction described above can also be seen as a $\mathcal{E}_{b+1}$-map by generalizing the idea that any unit is a $\mathcal{E}_{1}$-map.
Our statement however only concerns itself with the underlying map in $\V$, stripped of any algebraic structures.
For the purpose of proving \autoref{thm:invertibility_criterion}, we only need to understand when these maps are equivalences.
Because the forgetful functor $\EAlg{b+1}(\V) \to \V$ is conservative, this result is sufficient.

\begin{proof}
    Let $\overline{A}$ and $\overline{B}$ be the induced $\mathbb{R}_{\leq 0}$ and $\mathbb{R}_{\geq 0}$-algebras, as in \autoref{constr:actions_existence_dual}, applied to $\sigma(M)$.
    One can dualize the diagram \eqref{eq:W_lift} and use it to extend the diagram \eqref{eq:center_lift} to the following:
\[\begin{tikzcd}
	{\Mod_{\overline{A}}} \\
	\\
	\\
	{\cat{W}} &&&& {\Mod_{\overline{B}}} \\
	\\
	\\
	\V
	\arrow[""{name=0, anchor=center, inner sep=0}, curve={height=-6pt}, from=1-1, to=4-1]
	\arrow[""{name=1, anchor=center, inner sep=0}, "{-\otimes_{\overline{A}} M}", curve={height=-6pt}, from=1-1, to=4-5]
	\arrow[""{name=2, anchor=center, inner sep=0}, "{-\otimes \overline{A}}", curve={height=-6pt}, from=4-1, to=1-1]
	\arrow[""{name=3, anchor=center, inner sep=0}, "{-\otimes M}", curve={height=-6pt}, from=4-1, to=4-5]
	\arrow[""{name=4, anchor=center, inner sep=0}, curve={height=-6pt}, from=4-1, to=7-1]
	\arrow[""{name=5, anchor=center, inner sep=0}, "{\hom_{\Mod_{\overline{B}}}(M,-)}", curve={height=-6pt}, from=4-5, to=1-1]
	\arrow[""{name=6, anchor=center, inner sep=0}, "{\hom_{\Mod_{\overline{B}}}(M,-)}", curve={height=-6pt}, from=4-5, to=4-1]
	\arrow[""{name=7, anchor=center, inner sep=0}, "{\hom_{\Mod_{\overline{B}}}(M,-)}", curve={height=-6pt}, from=4-5, to=7-1]
	\arrow[""{name=8, anchor=center, inner sep=0}, "{-\otimes W}", curve={height=-6pt}, from=7-1, to=4-1]
	\arrow[""{name=9, anchor=center, inner sep=0}, "{-\otimes M}", curve={height=-6pt}, from=7-1, to=4-5]
	\arrow["\dashv"{anchor=center, rotate=-139}, draw=none, from=1, to=5]
	\arrow["\dashv"{anchor=center}, draw=none, from=2, to=0]
	\arrow["\dashv"{anchor=center, rotate=-90}, draw=none, from=3, to=6]
	\arrow["\dashv"{anchor=center}, draw=none, from=8, to=4]
	\arrow["\dashv"{anchor=center, rotate=-49}, draw=none, from=9, to=7]
\end{tikzcd}\]
    Under the downward forgetful functor $\cat{W} \to \V$, the unit of the top adjunction evaluated at $\overline{A}$ is equivalent to the induced map of monad of the top triangle, evaluated at $\mathbbm{1}_{\cat{W}}$, which in turn is equivalent to the induced map of monad of the outer triangle, evaluated at $\mathbbm{1}_{\cat{V}}$.
    By the above discussion, this unit can be identified with the canonical map to some center.
    
    It now suffices to identify the functor $\Mod_{\overline{A}} \xto[- \otimes_{\overline{A}} M] \Mod_{\overline{B}}$ with the one induced by the action described in \autoref{rem:en_bimodule_version} (where the ``$n$'' there is our ``$b$'').
    By construction, the functor $- \otimes_{\overline{A}} M$ above corresponds to the bimodule structure induced by the restriction to the last coordinate $\mathbb{R}_{\overline{1}} \xmono[\textnormal{cl}] \mathbb{R}^n_{\overline{n}}$:
\[\begin{tikzcd}
	&& X & {\mathbb{R}_{\overline{1}}} \\
	{\mathbb{R}^n_{\overline{k}}} & {\mathbb{R}^n_{\overline{k}, \sigma}} & {\mathbb{R}^n_{\overline{k}, \overline{\sigma}}} & {\mathbb{R}^n_{\overline{n}}}
	\arrow[two heads, from=1-3, to=1-4]
	\arrow["{\textnormal{cl}}", "{(i)}"', hook, from=1-3, to=2-3]
	\arrow["{\textnormal{cl}}", hook, from=1-4, to=2-4]
	\arrow["\sim"{description}, no head, from=2-2, to=2-1]
	\arrow[dashed, from=2-3, to=2-2]
	\arrow[two heads, from=2-3, to=2-4]
\end{tikzcd}\]
    where $X = \mathbb{R}_{\leq 0} \times \mathbb{R}^a \times \{0\}^b \bigsqcup_{\{0\}^{d+1}}  \mathbb{R}_{\geq 0} \times \{0\}^a \times \mathbb{R}^b \epi \mathbb{R}_{\overline{1}}$ and
    the inclusion $(i)$ is induced by:
    \begin{enumerate}
        \item permuting the coordinates by mapping $\{1, \ldots, a+1\} \mapsto \{ k \leq i < k+d \mid \sigma_i = \epsilon\}$, and mapping $\{a+2, \ldots, d+1\} \mapsto \{ k \leq i < k+d \mid \sigma_i = \eta\}$,
        \item if $a > 0$, flipping the sign of the second coordinate if the sign of the permutation is negative.
    \end{enumerate}
    We then have the following diagram (see \autoref{fig:XY} below for a visual interpretation of the upper part of the diagram):
\begin{equation}\label{eq:XY}
\begin{tikzcd}
	&&& {\mathbb{R}^{b+1}_{\overline{1}}} \\
	&& Y & {\left(\mathbb{R}_{\leq 0} \times \{0\}^b\right) \bigsqcup_{\{0\}^{b+1}} \left(\mathbb{R}_{\geq 0} \times \mathbb{R}^{b}\right)} \\
	{\mathbb{R}^n_{\overline{k}}} & {\mathbb{R}^d_{\overline{1}}} & X & {\mathbb{R}_{\overline{1}}} \\
	{\mathbb{R}^n_{\overline{k}}} & {\mathbb{R}^n_{\overline{k}, \sigma}} & {\mathbb{R}^n_{\overline{k}, \overline{\sigma}}} & {\mathbb{R}^n_{\overline{n}}}
	\arrow[two heads, from=2-3, to=1-4]
	\arrow["{\textnormal{cl}}"', hook', from=2-3, to=3-2]
	\arrow["{\textnormal{cl}}"', hook, from=2-4, to=1-4]
	\arrow[two heads, from=2-4, to=3-4]
	\arrow["\sim"{marking, allow upside down}, "{(ii)}"', no head, from=3-1, to=4-1]
	\arrow["{\textnormal{cl}}"', hook', from=3-2, to=3-1]
	\arrow["{\textnormal{cl}}"', hook', from=3-3, to=2-3]
	\arrow[two heads, from=3-3, to=2-4]
	\arrow["{\textnormal{cl}}", hook', from=3-3, to=3-2]
	\arrow[two heads, from=3-3, to=3-4]
	\arrow["{\textnormal{cl}}", "{(i)}"', hook, from=3-3, to=4-3]
	\arrow["{\textnormal{cl}}", hook, from=3-4, to=4-4]
	\arrow["\sim"{description}, no head, from=4-2, to=4-1]
	\arrow[dashed, from=4-3, to=4-2]
	\arrow[two heads, from=4-3, to=4-4]
\end{tikzcd}\end{equation}
    where $Y = (\{0\}^{a+1} \times \mathbb{R}^b \subset \{0\} \times \mathbb{R}^d \subset \mathbb{R}_{\geq 0} \times \mathbb{R}^d) \bigsqcup_{\{0\}^{a+1} \times \mathbb{R}^b} (\mathbb{R}_{\leq 0} \times \{0\}^a \times \mathbb{R}^b)$ is the intermediate space in the statement of the result,
    and $(ii)$ is the stratified automorphism induced by the same permutation and sign-flip as $(i)$.
    In particular, $(ii)$ is isotopic to the identity.
    Using notably \autoref{rem:vector_fields}, each face of the diagram can be seen to satisfy the condition of \autoref{lem:commutative_diagrams}, resulting in a commutative diagram of functors between categories of algebra.
    Using the isotopy $(ii) \sim \mathbbm{1}$, the left/top side corresponds to the action described in the statement.
    The right side corresponds to the construction of \autoref{rem:en_bimodule_version}.
    The bottom corresponds to the original description coming from \autoref{def:bar_handlebody_strata}. 
    Putting all these together, we get the desired identification.
\end{proof}

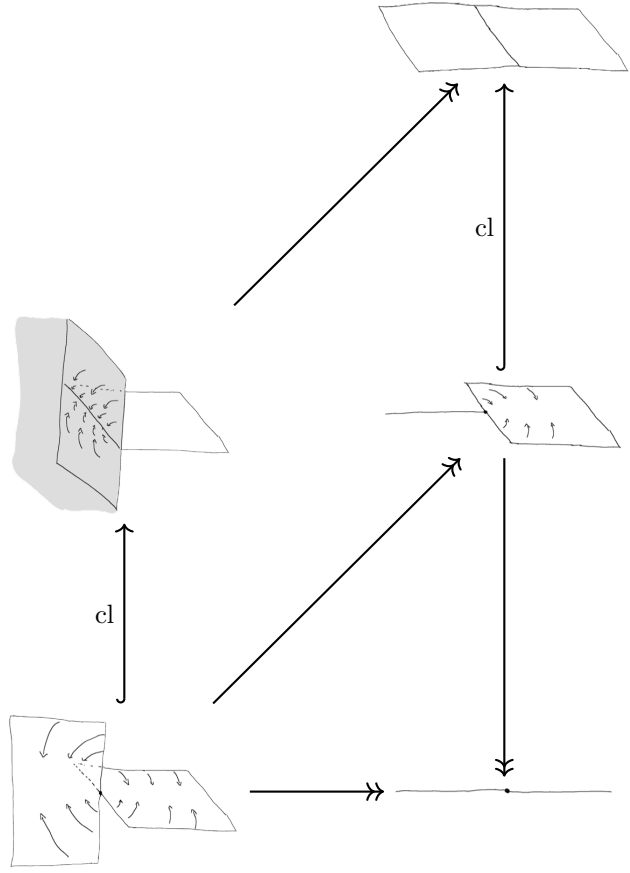
\begin{figure}[H]
\centering
\begin{tikzpicture}
    \def\scaleFactor{0.2}  
    \def\imageSeparation{5}  
    
    \node (img1) at (0, 0) {\includegraphics[scale=\scaleFactor]{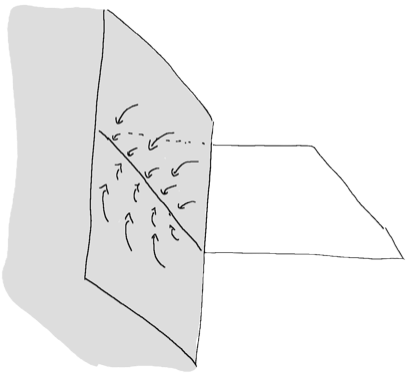}};

    \node (img2) at (0, -\imageSeparation) {\includegraphics[scale=\scaleFactor]{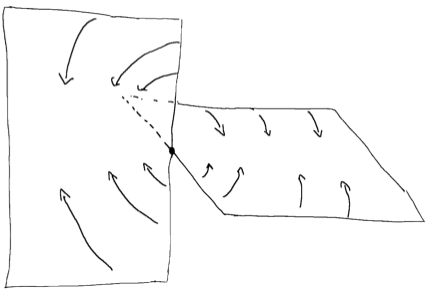}};
    
    \node (img3) at (\imageSeparation, \imageSeparation) {\includegraphics[scale=\scaleFactor]{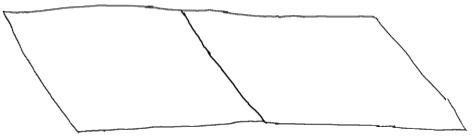}};

    \node (img4) at (\imageSeparation, 0) {\includegraphics[scale=\scaleFactor]{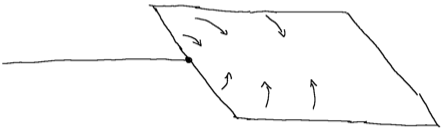}};

    \node (img5) at (\imageSeparation, -\imageSeparation) {\includegraphics[scale=\scaleFactor]{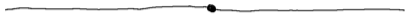}};
    
    \draw[{Hooks[right]}->, thick] (img2) -- (img1) node[midway, left] {cl};
    \draw[{Hooks[right]}->, thick] (img4) -- (img3) node[midway, left] {cl};
    \draw[->>, thick] (img1) -- (img3);
    \draw[->>, thick] (img2) -- (img4);
    \draw[->>, thick] (img2) -- (img5);
    \draw[->>, thick] (img4) -- (img5);
\end{tikzpicture}
\caption{
    Visualization of the diagram \eqref{eq:XY}.
}
\label{fig:XY}
\end{figure}

\section{Proof of the Criteria}\label{sec:proofs}

In this final section, we prove our main Theorems (\autoref{thm:full_dualizability_criterion} and \autoref{thm:invertibility_criterion}).
The core of the argument is inductive: as proved in \autoref{sec:adjoints}, to show that some object is fully-dualizable (resp. invertible), knowing that it already is dualizable, it suffices to show its (co)units are fully-dualizable (resp. invertible).
The tools developed in \autoref{sec:factorization_homology} then allow to precisely describe what happens when one iterates the ``$f \mapsto \textnormal{ (co)units of } f$'' construction in the higher Morita categories.
Combining both elements, we immediately obtain our criteria.

\subsection{Theorem A}\label{subsec:thm_A}

Before heading into the proof, we establish a small well-known lemma.
The idea behind this lemma is that the data exhibiting morphisms as higher full-adjoints is highly redundant.
This property is sometimes known as ambidexterity in the literature.

\begin{lemma}[Redundancy of Duality]\label{lem:redundancy}
    Let $\cat{C} \inset \Cat{3}$ and suppose we have an adjunction $L \in \begin{tikzcd}
        + & -
        \arrow[""{name=0, anchor=center, inner sep=0}, curve={height=-6pt}, from=1-1, to=1-2]
        \arrow[""{name=1, anchor=center, inner sep=0}, curve={height=-6pt}, from=1-2, to=1-1]
        \arrow["\dashv"{anchor=center, rotate=-90}, draw=none, from=0, to=1]
    \end{tikzcd} \in R $ such that the units and counits $\epsilon, \eta$ admit both left and right adjoints.
    Then $\epsilon$ and $\eta$ are fully-adjoint.
\end{lemma}
\begin{proof}
    Let $\epsilon^r \dashv \epsilon \dashv \epsilon^l$ and $\eta^r \dashv \eta \dashv \eta^l$.
    By taking left adjoints to each triangle identities, we can see that $\epsilon^l$ and $\eta^l$ are respectively a unit and a counit for $R \dashv L$.
    Similarly, by taking right adjoints to each triangle identities, we can see that $\epsilon^r$ and $\eta^r$ are respectively a unit and a counit for the same adjunction $R \dashv L$.
    By uniqueness of adjoints, we get $\epsilon^l \simeq \epsilon^r$ and $\eta^l \simeq \eta^r$, which then implies the result.
\end{proof}

We can now prove our main Theorem.

\begin{construction}[Actions in \autoref{thm:full_dualizability_criterion}]\label{constr:canonical_actions_thm_A}
    Let $\V \inset \EAlg{\infty}(\Cat{1}^\geom)$ be a symmetric monoidal $(\infty,1)$-category with geometric realizations and compatible tensor product.
    Let $0 \leq k < n$ and $M$ be a $k$-morphism in $\Mor_n(\cat{V})$.
    Let $A$ (resp. $B$) be the (co)domain of $M$ if $1 \leq k$, or otherwise be the identity $\mathbbm{1}_\V$ if $k=0$.
    Recall that by convention, we let $S^{-1} = \emptyset$.
    For $0 \leq i \leq n-k$, the homological maps
    \[\begin{tikzcd}
        {\mathbb{R}^n_{\overline{k}}} & {\{0\}^{k-1} \times \mathbb{R}_{\geq 0} \times \mathbb{R}^i \times \{0\}^{n-k-i}} & {\mathbb{R}_{\geq 0}}
        \arrow["{\textnormal{cl}}", hook', from=1-2, to=1-1]
        \arrow["r", two heads, from=1-2, to=1-3]
    \end{tikzcd}\]
    (resp. with $\leq$) where $r$ is the radius function, induce a $\cat{E}_1$-algebra structure on $\int_{(S^{i-1}, D^i)} (M, A)$ (resp. $\int_{(S^{i-1}, D^i)} (M, B)$) acting on $M$.
    We denote these the \textbf{canonical $\int_{(S^{i-1}, D^i)} (M, A) \curvearrowright M$ and $M \curvearrowleft \int_{(S^{i-1}, D^i)} (M, B)$ actions}.
    See \autoref{fig:canonical_actions_thm_A} below for a visual interpretation of these actions.
\end{construction}

\begin{theorem}[Full-Dualizability Criterion]\label{thm:full_dualizability_criterion}
    Let $\V \inset \EAlg{\infty}(\Cat{1}^\geom)$, $0 \leq k < n$ and $M \in c_k \to \Mor_n(\cat{V})$ be as in \autoref{constr:canonical_actions_thm_A} above.
    Then $M$ is fully-dualizable if and only if, for all $0 \leq i \leq n-k$, the canonical actions
    \begin{equation*}[rCl]
        \int_{(S^{i-1}, D^i)} (M, A) \curvearrowright &M& \\
        &M& \curvearrowleft \int_{(S^{i-1}, D^i)} (M, B)
    \end{equation*}
    exhibit $M$ as a dualizable module, where $A$ (resp. $B$) is the (co)domain of $M$ (see \autoref{constr:canonical_actions_thm_A} above for details on these actions).
\end{theorem}

Note that when $k = 0$, we recover the usual criterion first stated by Lurie \cite[Rem. $4.1.27$]{Lur09a}:
\[
    \forall 0 \leq i \leq n, \qquad M \text{ is a dualizable } \int_{S^{i-1}} M\text{-module}.
\]

\begin{figure}[H]
\centering
\begin{tikzpicture}
    \def\scaleFactor{0.4}  

    \node (img1) at (0, 0) {\includegraphics[scale=\scaleFactor]{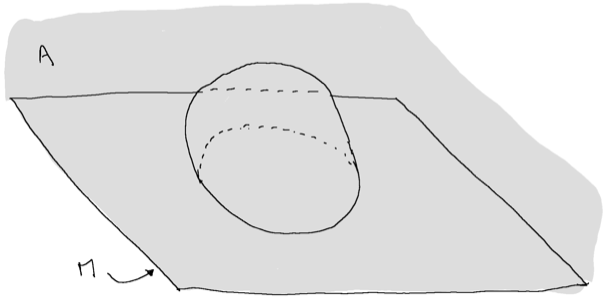}};
\end{tikzpicture}
\caption{
    Visualization of the action on the left-hand side of \eqref{eq:action} for $i = 2$.
    The background space is $\mathbb{R}^2 \times \mathbb{R}_{\geq 0}$ with $A$ supported on the $3$-dimensional strata (hence seen as an $\mathcal{E}_3$-algebra) and $M$ supported on the $2$-dimensional strata (hence seen as an $\mathcal{E}_3$-algebra acted on in an $\mathcal{E}_1$-one-sided-way by $A$).
    The $\mathcal{E}_1$-algebra of interest is then the factorization homology of an upper-half dome, with the $\mathcal{E}_1$-structure given by the radial symmetry, and it acts on $M$ pictured at the center of the bottom plane.
    This factorization homology can be interpreted as a relative form of Hochschild Homology.
}
\label{fig:canonical_actions_thm_A}
\end{figure}
\begin{proof}
    Let $d = n-k$.
    Firstly, note that the actions described in \autoref{constr:canonical_actions_thm_A} correspond precisely to all the actions involved in \autoref{prop:existence_dual} applied to the collection of $n$-morphisms given by: $\sigma(M)$ and $\sigma(M^\circ)$ for $\sigma \inset \{\epsilon, \eta\}^d$ a $d$-duality datum. 
    Indeed, similary to the proof of \autoref{prop:identifying_units}, one has a commutative diagram of homological maps
\[\begin{tikzcd}
	{\mathbb{R}^n_{\overline{k}}} & {\mathbb{R}_{\geq 0} \times \mathbb{R}^i} && {\mathbb{R}_{\geq 0}} \\
	{\mathbb{R}^n_{\overline{k}}} & {\mathbb{R}^n_{\overline{k}, \sigma}} & {\mathbb{R}^n_{\overline{k}, \overline{\sigma}}} & {\mathbb{R}^n_{\overline{n}}}
	\arrow["\sim"{marking, allow upside down}, "{(I)}"', no head, from=1-1, to=2-1]
	\arrow[hook', "{\textnormal{cl}}"', from=1-2, to=1-1]
	\arrow["{\textnormal{r}}", two heads, from=1-2, to=1-4]
	\arrow["{(II)}", "{\textnormal{cl}}"', hook, from=1-2, to=2-3]
	\arrow[hook, "{\textnormal{cl}}", from=1-4, to=2-4]
	\arrow["\sim"{marking, allow upside down}, no head, from=2-2, to=2-1]
	\arrow[dashed, from=2-3, to=2-2]
	\arrow[two heads, from=2-3, to=2-4]
\end{tikzcd}\]
    where $i = |\{\sigma = \eta\}|$ (see \autoref{def:duality_data} for details on this notation), $(II)$ is the inclusion of the coordinates $\{\sigma = \eta\} \subset \{k, \ldots, k+d-1\}$ and $(I)$ is defined as follows:
    \begin{enumerate}
        \item if $i = n-k$, then $(I)$ is the identity,
        \item if $i < n-k$, then $(I)$ is an automorphism that permutes the coordinates by sending $\{\sigma = \eta\} + 1$ to $\{k+1, \ldots, k+d\}$, and flips the sign of one of the remaining coordinates if necessary to ensure the orientation is preserved. 
    \end{enumerate}
    In any case, $(I)$ is isotopic to the identity.
    By \autoref{rem:vector_fields}, each face of the diagram satisfies the condition of \autoref{lem:commutative_diagrams}, resulting in a commutative diagram of functors between categories of algebra.
    This identifies the actions as claimed.

    Hence, it suffices secondly to show by downward induction on $k$ the following result:
    $M$ is fully-dualizable if and only if, for all $\sigma \inset \{\epsilon, \eta\}^d$ a $d$-duality datum, $\sigma(M)$ and $\sigma(M^\circ)$ have both left and right adjoints, when seen as $n$-morphisms in $\Mor_n(\cat{V})$.

    \begin{enumerate}
        \item \textbf{Case $k = n-1$:} By \autoref{lem:redundancy}, the condition above is equivalent to having both left and right adjoints for each of the units and counits of the full-$\{n-1\}$-adjunction.
        By \autoref{cor:units_in_morita}, these units and counits are described by the algebras $\eta(M)$, $\epsilon(M)$, $\eta(M^\circ)$, $\epsilon(M^\circ)$.
        By \autoref{prop:existence_dual}, this is equivalent to the condition in the statement of the theorem.

        \item \textbf{Inductive Step:} Note first that by \autoref{prop:free_sub_compose} and \autoref{cor:lifting_once_fully_dualizable}, because $\Mor_n(\V)$ has $\{1, \ldots, n-1\}$-adjoints, $M$ is fully-dualizable if and only if we have a lift of the dashed arrow, if and only if we have a lift of each of the dotted arrows:
\[\begin{tikzcd}
	& {\Sub^{\{k+1, \ldots, n\}}(\Mor_n(\cat{V}))} & {\Mor_n(\cat{V})} \\
	{c_{k+1}} \\
	& {c_k} & {\Free^{\{k\}}(c_k)} \\
	{c_{k+1}}
	\arrow[hook, from=1-2, to=1-3]
	\arrow[dotted, from=2-1, to=1-2]
	\arrow["{\epsilon(f^{\vee n})}"{description, pos=0.6}, from=2-1, to=3-3]
	\arrow[from=3-2, to=1-2]
	\arrow[two heads, from=3-2, to=3-3]
	\arrow[dashed, from=3-3, to=1-2]
	\arrow["M"', from=3-3, to=1-3]
	\arrow[dotted, from=4-1, to=1-2]
	\arrow["{\eta(f^{\vee n})}"{description}, from=4-1, to=3-3]
\end{tikzcd}\]
    In other words, $M$ is fully-dualizable if and only if all the units and counits involved in the full-$\{k\}$-adjunction are themselves fully-dualizable.
    As above, by \autoref{cor:units_in_morita}, these units and counits are described by the algebras $\eta(M)$, $\epsilon(M)$, $\eta(M^\circ)$, $\epsilon(M^\circ)$.
    By the inductive hypothesis, duality is then equivalent to:
    for all $\sigma \inset \{\epsilon, \eta\}^{d-1}$ a $(d-1)$-duality datum, $\sigma(\eta(M))$, $\sigma(\epsilon(M))$, $\sigma(\eta(M^\circ))$, $\sigma(\epsilon(M^\circ))$, $\sigma(\eta(M)^\circ)$, $\sigma(\epsilon(M)^\circ)$, $\sigma(\eta(M^\circ)^\circ)$ and $\sigma(\epsilon(M^\circ)^\circ)$ all have left and right adjoints, when seen as $n$-morphisms in $\Mor_n(\cat{V})$.
    Using \autoref{lem:associativity} and the exchange property of $\circ$ with $\eta \leftrightarrow \epsilon$, we can conclude the proof.
    \end{enumerate}
\end{proof}

In particular, dualizable module actions are well-understood in the category $\Spe$ of spectra, and we get this immediate corollary.

\begin{corollary}[Full-Dualizability of Ring Spectra]
    Let $0 \leq k < n$ and $M \in c_k \to \Mor_n(\Spe)$ be as in \autoref{constr:canonical_actions_thm_A} above with $\V = \Spe$.
    Then $M$ is fully-dualizable if and only if, for all $1 \leq i \leq n-k$, the canonical actions
    \begin{equation*}[rCl]\label{eq:action}
        \int_{(S^{i-1}, D^i)} (M, A) \curvearrowright &M& \\
        &M& \curvearrowleft \int_{(S^{i-1}, D^i)} (M, B)
    \end{equation*}
    exhibit $M$ as a perfect or compact module, where $A$ (resp. $B$) is the (co)domain of $M$ (see \autoref{constr:canonical_actions_thm_A} for details on these actions).
\end{corollary}
\begin{proof}
    This is a direct consequence of \autoref{thm:full_dualizability_criterion} and \autoref{cor:existence_dual_spectra}.
\end{proof}

\begin{remark}[Characterizing Full-Dualizability in Other Higher Categories]
    Implicitly, the proof uses that: if $f$ is a fully-$\{k, \ldots, l\}$-dualizable $k$-morphism for $k \leq l$, full-$\{k, \ldots, l+1\}$-dualizability is equivalent to all $\sigma(f)$ having left adjoints for $\sigma \inset \{\epsilon, \eta\}^{l-k+1}$ a $(l-k+1)$-duality datum.
    This can be proved independently in any higher category with the tools of \autoref{sec:adjoints} following the same argument.
    It would simply remain to show a few more ``redundancy'' lemmas.
    In our case, the Morita category is built with natural symmetries $\Aut(X) \curvearrowright \Alg{X}$ (see \autoref{def:operads}), which allows us to easily access this redundancy.
\end{remark}

\begin{remark}[Other Proof Strategies]
    There are many possible proof strategies for this result, from closest to farthest from the one presented here:
    \begin{enumerate}
        \item One could do an upward induction on $k$ instead, using some right associativity lemma for the pinch constructions.
        This lemma is significantly heavier to spell out as the stratifications involved are more complicated, but the core idea is the same.

        \item One could also imagine a more elementary proof of the result above without explicit description of the pinch maps, by naively writing an inductive argument.
        This would however not translate well to the proof of Theorem B, where the explicit description of the pinch maps is very helpful to elucidate \autoref{prop:identifying_units}.

        \item One could also try to write the $(\infty,n)$-functor $\Bord_n^{\textnormal{fr}} \to \Mor_n(\V)$ and use handlebody decomposition to identify which actions on the top $n$-morphism level need to be dualizable.
        This is a much more technical approach.
        Such a functor was constructed in \cite{Sch14} for a pointed version of $\Mor_n(\V)$ with the language of factorization algebras and relative categories.
        A proof following this approach would involve re-proving many technical results already present in \cite{AFT17a, AFR18, AFT17} under a different formalism, or showing an equivalence of both formalisms.
        We avoid this by using a crucial idea: deciding if a morphism is fully-dualizable is not necessarily a homotopy-coherent problem.
    \end{enumerate}
\end{remark}

\subsection{Theorem B}\label{subsec:thm_B}

Similarly to the proof of Theorem $A$, we start with a redundancy lemma.
Only this time, we are interested in the redundancy of invertibility properties.

\begin{lemma}[Redundancy of Invertibility]\label{lem:redundancy_invertibility}
    Let $\cat{C} \inset \Cat{2}$ and suppose we have two parallel adjunctions
\[\begin{tikzcd}
	\bullet && \bullet
	\arrow[""{name=0, anchor=center, inner sep=0}, "f"{description}, from=1-1, to=1-3]
	\arrow[""{name=1, anchor=center, inner sep=0}, curve={height=-18pt}, from=1-3, to=1-1]
	\arrow[""{name=2, anchor=center, inner sep=0}, curve={height=18pt}, from=1-3, to=1-1]
	\arrow["\dashv"{anchor=center, rotate=-90}, draw=none, from=0, to=1]
	\arrow["\dashv"{anchor=center, rotate=-90}, draw=none, from=2, to=0]
\end{tikzcd}\]
    in $\cat{C}$ such that the units are both equivalences.
    Then $f$ is an equivalence.
\end{lemma}
\begin{proof}
    Note that equivalences and adjoints are detected in the homotopy $(2,2)$-category.
    Without loss of generality, one can then assume that $\cat{C}$ is a $(2,2)$-category. 
    In this case, one can use the bicategorical Yoneda lemma \cite[Lem. $8.3.12$]{JY20} to reduce to the case where $\cat{C} = \mathbf{Cat}_{(1,1)}$.
    In $\mathbf{Cat}_{(1,1)}$, the unit of the bottom adjunction being a natural isomorphism implies that $f$ is fully faithful, which in turn implies that the counit of the top adjunction is a natural isomorphism.
    Hence, the top adjunction exhibits $f$ as an (adjoint) equivalence.
\end{proof}

Using the above lemma, we can then prove our second criterion.

\begin{theorem}[Invertibility Criterion]\label{thm:invertibility_criterion}
    Let $\V \inset \EAlg{\infty}(\Cat{1}^\geom)$ be a symmetric monoidal $(\infty,1)$-category with geometric realizations and compatible tensor product.
    Let $0 \leq k < n$ and $M$ be a $k$-morphism in $\Mor_n(\cat{V})$.
    Then $M$ is invertible if and only if it is fully-dualizable and, for all $1 \leq i \leq n-k$, the universal $\mathcal{E}_{n-k-i+1}$-algebra maps
    \[
        \int_{(S^{i-1}, D^i)} (M, A) \longeq \HC{n-k-i}^B(M) \\
        \int_{(S^{i-1}, D^i)} (M, B) \longeq \HC{n-k-i}^A(M)
    \]
    are equivalences, where $A$ (resp. $B$) is the (co)domain of $M$ (see \autoref{constr:canonical_action_units} for details on these maps).
\end{theorem}

Fundamentally, one can interpret invertibility of these maps as saying: $(i-1)^\textnormal{th}$-higher Hochschild Homology of $M$ relative to $A$ computes the universal $\mathcal{E}_{k-n-i+1}$-algebra acting on $M$ in a $B$-linear way (and vice-versa).
Note that once again, we recover Lurie's criterion for $k=0$, and more familiarly, we recover the usual conditions for an Azumaya algebra when $n=1$ and $k=0$: $M$ is dualizable and 
\begin{equation*}
    k \longeq \HC{1}(M) \\
    M^{\textnormal{op}} \otimes_k M \longeq \End_k(M)
\end{equation*}

\begin{proof}
    By \autoref{lem:free_contractible}, the functor $\Free^{\{k, \ldots, n\}}(c_k) \epi c_{k-1}$ is epimorphic and correspond under precomposition to some $[c_k, \Mor_n(\cat{V})]_{\simeq} \mono [\Free^{\{k, \ldots, n\}}(c_k), \Mor_n(\cat{V})]$.
    Hence, any invertible $k$-morphism is automatically fully-dualizable.
    We can then assume $M$ is fully-dualizable, and show invertibility is equivalent to all these maps of algebra being invertible.
    
    Let $d = n - k$ (which can be zero).
    Similarly to Theorem A, now using \autoref{prop:identifying_units}, we first observe that invertibility of all these maps is precisely equivalent to the statement:
    for all $\sigma \inset \{\epsilon, \eta\}^d$ a $d$-duality datum, the unit of $\sigma(M) \dashv \sigma(M)^\vee$ and of $\sigma(M)^\circ \dashv (\sigma(M)^\circ)^\vee$ are invertible.
    As a consequence, it suffices to show by downward induction on $k$ the equivalence between invertibility of $M$ and this condition.
    \begin{enumerate}
        \item \textbf{Base Case $k = n$:} this is an immediate consequence of \autoref{lem:redundancy_invertibility} above.
        \item \textbf{Inductive Step:} By \autoref{lem:adj_contractible} and \autoref{cor:units_in_morita}, $M$ is invertible if and only if $\epsilon(M)$ and $\eta(M)$ are invertible as $(k+1)$-morphisms in $\Mor_n(\cat{V})$.
        The statement then follows from the inductive hypothesis and \autoref{lem:associativity}.
    \end{enumerate}
\end{proof}

\bibliographystyle{alpha}
\bibliography{jabref}

@Book{JY20,
  author        = {Johnson, Niles and Yau, Donald},
  publisher     = {Oxford University Press},
  title         = {2-{D}imensional {C}ategories},
  year          = {2021},
  address       = {Oxford},
  isbn          = {0198871376},
  month         = jan,
  archiveprefix = {arXiv},
  date          = {2020-02-14},
  doi           = {10.1093/oso/9780198871378.001.0001},
  ean           = {9780198871378},
  eprint        = {2002.06055},
  eprintclass   = {math.CT},
  eprinttype    = {arXiv},
  pagetotal     = {640},
  primaryclass  = {math.CT},
  url           = {https://oxford.universitypressscholarship.com/10.1093/oso/9780198871378.001.0001/oso-9780198871378},
}

@Book{Lur09,
  author    = {Lurie, Jacob},
  publisher = {Princeton University Press},
  title     = {Higher {T}opos {T}heory. ({AM-170})},
  year      = {2009},
  isbn      = {9780691140490},
  eprint    = {math/0608040},
}

@Misc{Lur17,
  author       = {Lurie, Jacob},
  howpublished = {\url{https://www.math.ias.edu/~lurie/papers/HA.pdf}},
  month        = sep,
  title        = {Higher {A}lgebra},
  year         = {2017},
  url          = {https://www.math.ias.edu/~lurie/papers/HA.pdf},
}

@Article{AFR18,
  author        = {Ayala, David and Francis, John and Rozenblyum, Nick},
  journal       = {Advances in Mathematics},
  title         = {Factorization homology I: Higher categories},
  year          = {2018},
  issn          = {0001-8708},
  month         = jul,
  pages         = {1042--1177},
  volume        = {333},
  archiveprefix = {arXiv},
  doi           = {10.1016/j.aim.2018.05.031},
  eprint        = {1504.04007},
  primaryclass  = {math.AT},
  publisher     = {Elsevier BV},
  url           = {https://www.sciencedirect.com/science/article/pii/S0001870818302068},
}

@Article{AFT17,
  author        = {Ayala, David and Francis, John and Tanaka, Hiro Lee},
  journal       = {Selecta Mathematica},
  title         = {Factorization Homology of Stratified Spaces},
  year          = {2017},
  issn          = {1420-9020},
  month         = apr,
  number        = {1},
  pages         = {293--362},
  volume        = {23},
  archiveprefix = {arXiv},
  date          = {2017/01/01},
  doi           = {10.1007/s00029-016-0242-1},
  eprint        = {1409.0848},
  id            = {Ayala2017},
  isbn          = {1420-9020},
  primaryclass  = {math.AT},
  publisher     = {Springer Science and Business Media LLC},
  url           = {https://doi.org/10.1007/s00029-016-0242-1},
}

@Article{AFR18a,
  author        = {Ayala, David and Francis, John and Rozenblyum, Nick},
  journal       = {Journal of the European Mathematical Society},
  title         = {A Stratified Homotopy Hypothesis},
  year          = {2018},
  issn          = {1435-9863},
  month         = dec,
  number        = {4},
  pages         = {1071--1178},
  volume        = {21},
  archiveprefix = {arXiv},
  doi           = {10.4171/jems/856},
  eprint        = {1502.01713},
  primaryclass  = {math.AT},
  publisher     = {European Mathematical Society - EMS - Publishing House GmbH},
}

@Article{Lur09a,
  author        = {Lurie, Jacob},
  title         = {On the {C}lassification of {T}opological {F}ield {T}heories},
  year          = {2009},
  archiveprefix = {arXiv},
  eprint        = {0905.0465},
  primaryclass  = {math.CT},
  url           = {https://arxiv.org/abs/0905.0465},
}

@Article{BS20,
  author        = {Barwick, Clark and Schommer-Pries, Christopher},
  journal       = {Journal of the American Mathematical Society},
  title         = {On the {U}nicity of the {H}omotopy {T}heory of {H}igher {C}ategories},
  year          = {2020},
  issn          = {1088-6834},
  month         = apr,
  number        = {4},
  pages         = {1011--1058},
  volume        = {34},
  archiveprefix = {arXiv},
  doi           = {https://doi.org/10.1090/jams/972},
  eprint        = {1112.0040},
  primaryclass  = {math.AT},
  publisher     = {American Mathematical Society (AMS)},
  url           = {https://arxiv.org/abs/1112.0040},
}

@Article{Rez10,
  author    = {Rezk, Charles},
  journal   = {Geometry \& Topology},
  title     = {A {C}artesian {P}resentation of {W}eak $n$–{C}ategories},
  year      = {2010},
  issn      = {1465-3060},
  month     = jan,
  number    = {1},
  pages     = {521--571},
  volume    = {14},
  doi       = {10.2140/gt.2010.14.521},
  publisher = {Mathematical Sciences Publishers},
}

@Unpublished{Joy97,
  author = {Joyal, Andr{\'e}},
  title  = {Disks, {Duality} and $\Theta$-{C}ategories},
  year   = {1997},
}

@Article{AF18,
  author        = {Ayala, David and Francis, John},
  title         = {Flagged {H}igher {C}ategories},
  year          = {2018},
  archiveprefix = {arXiv},
  eprint        = {1801.08973},
  primaryclass  = {math.CT},
  url           = {https://arxiv.org/abs/1801.08973},
}

@Article{SS86,
  author    = {Schanuel, Stephen and Street, Ross},
  journal   = {Cahiers de {T}opologie et {G}{\'e}om{\'e}trie {D}iff{\'e}rentielle {C}at{\'e}goriques},
  title     = {The {F}ree {A}djunction},
  year      = {1986},
  number    = {1},
  pages     = {81--83},
  volume    = {27},
  language  = {en},
  mrnumber  = {845410},
  publisher = {Dunod \'editeur, publi\'e avec le concours du CNRS},
  url       = {https://www.numdam.org/item/CTGDC_1986__27_1_81_0/},
  zbl       = {0592.18002},
}

@Article{RV16,
  author    = {Riehl, Emily and Verity, Dominic},
  journal   = {Advances in Mathematics},
  title     = {Homotopy {C}oherent {A}djunctions and the {F}ormal {T}heory of {M}onads},
  year      = {2016},
  issn      = {0001-8708},
  month     = jan,
  pages     = {802-888},
  volume    = {286},
  doi       = {https://doi.org/10.1016/j.aim.2015.09.011},
  publisher = {Elsevier BV},
  url       = {https://www.sciencedirect.com/science/article/pii/S0001870815003503},
}

@PhdThesis{Sch14,
  author    = {Scheimbauer, Claudia Isabella},
  school    = {ETH Zurich},
  title     = {Factorization {H}omology as a {F}ully {E}xtended {T}opological {F}ield {T}heory},
  year      = {2014},
  address   = {Zurich},
  note      = {Diss., Eidgenössische Technische Hochschule ETH Zürich, Nr. 22130.},
  type      = {Doctoral Thesis},
  copyright = {In Copyright - Non-Commercial Use Permitted},
  doi       = {10.3929/ethz-a-010399715},
  language  = {en},
  publisher = {ETH Zurich},
  size      = {1 Band},
}

@Article{Toe12,
  author    = {To{\"e}n, Bertrand},
  journal   = {Inventiones mathematicae},
  title     = {Derived {A}zumaya {A}lgebras and {G}enerators for {T}wisted {D}erived {C}ategories},
  year      = {2012},
  issn      = {1432-1297},
  month     = jan,
  number    = {3},
  pages     = {581--652},
  volume    = {189},
  doi       = {https://doi.org/10.1007/s00222-011-0372-1},
  publisher = {Springer Science and Business Media LLC},
}

@Article{AFT17a,
  author    = {Ayala, David and Francis, John and Tanaka, Hiro Lee},
  journal   = {Advances in Mathematics},
  title     = {Local {S}tructures on {S}tratified {S}paces},
  year      = {2017},
  issn      = {0001-8708},
  month     = feb,
  pages     = {903-1028},
  volume    = {307},
  doi       = {https://doi.org/10.1016/j.aim.2016.11.032},
  publisher = {Elsevier BV},
  url       = {https://www.sciencedirect.com/science/article/pii/S0001870816316097},
}

@Article{Hau23,
  author        = {Haugseng, Rune},
  title         = {Some {R}emarks on {H}igher {M}orita {C}ategories},
  year          = {2023},
  archiveprefix = {arXiv},
  eprint        = {2309.09761},
  primaryclass  = {math.CT},
  url           = {https://arxiv.org/abs/2309.09761},
}

@Article{Hau17,
  author    = {Haugseng, Rune},
  journal   = {Geometry {\&} Topology},
  title     = {The {H}igher {M}orita {C}ategory of $\mathbb{E}_{n}$–{A}lgebras},
  year      = {2017},
  issn      = {1465-3060},
  month     = may,
  number    = {3},
  pages     = {1631 -- 1730},
  volume    = {21},
  doi       = {10.2140/gt.2017.21.1631},
  publisher = {MSP},
  url       = {https://doi.org/10.2140/gt.2017.21.1631},
}

@Article{BS24,
  author        = {Barkan, Shaul and Steinebrunner, Jan},
  title         = {Segalification and the {B}oardman-{V}ogt {T}ensor {P}roduct},
  year          = {2024},
  archiveprefix = {arXiv},
  eprint        = {2301.08650},
  primaryclass  = {math.AT},
  url           = {https://arxiv.org/abs/2301.08650},
}

@Article{DS11,
  author    = {Dugger, Daniel and Spivak, David},
  journal   = {Algebraic {\&} Geometric Topology},
  title     = {Rigidification of {Q}uasi-{C}ategories},
  year      = {2011},
  issn      = {1472-2747},
  month     = jan,
  number    = {1},
  pages     = {225--261},
  volume    = {11},
  doi       = {http://dx.doi.org/10.2140/agt.2011.11.225},
  publisher = {Mathematical Sciences Publishers},
}

@Article{Hau17a,
  author    = {Haugseng, Rune},
  journal   = {Proceedings of the American Mathematical Society},
  title     = {On the {E}quivalence between {$\Theta _{n}$}-{S}paces and {I}terated {S}egal {S}paces},
  year      = {2017},
  issn      = {1088-6826},
  month     = dec,
  number    = {4},
  pages     = {1401--1415},
  volume    = {146},
  doi       = {http://dx.doi.org/10.1090/proc/13695},
  eprint    = {https://arxiv.org/abs/1604.08480},
  publisher = {American Mathematical Society (AMS)},
}

@Article{AF24,
  author    = {Ayala, David and Francis, John},
  title     = {The {T}angle {H}ypothesis: {D}imension {1}},
  year      = {2024},
  copyright = {Creative Commons Attribution 4.0 International},
  doi       = {https://doi.org/10.48550/arXiv.2410.23965},
  publisher = {arXiv},
}

@Article{Rez01,
  author    = {Rezk, Charles},
  journal   = {Transactions of the American Mathematical Society},
  title     = {A {M}odel for the {H}omotopy {T}heory of {H}omotopy {T}heory},
  year      = {2001},
  issn      = {00029947, 10886850},
  number    = {3},
  pages     = {973--1007},
  volume    = {353},
  publisher = {American Mathematical Society},
  url       = {http://www.jstor.org/stable/221843},
  urldate   = {2025-10-06},
}

@PhdThesis{Bar05,
  author   = {Barwick, Clark},
  title    = {{$(\infty, n)$}-{C}at as a {C}losed {M}odel {C}ategory},
  year     = {2005},
  note     = {Copyright - Database copyright ProQuest LLC; ProQuest does not claim copyright in the individual underlying works; Last updated - 2023-03-03},
  isbn     = {978-0-542-00534-3},
  journal  = {ProQuest Dissertations and Theses},
  language = {English},
  pages    = {48},
  url      = {http://turing.library.northwestern.edu/login?url=https://www.proquest.com/dissertations-theses/∞-i-n-cat-as-closed-model-category/docview/305445747/se-2},
}

@Misc{kerodon,
  author       = {Jacob Lurie},
  howpublished = {\url{https://kerodon.net}},
  title        = {Kerodon},
  year         = {2018},
}

@Article{Ram24a,
  author        = {Ramzi, Maxime},
  title         = {Locally {R}igid {$\infty$}-{C}ategories},
  year          = {2024},
  archiveprefix = {arXiv},
  eprint        = {2410.21524},
  primaryclass  = {math.CT},
  url           = {https://arxiv.org/abs/2410.21524},
}

@Article{HHLN24,
  author        = {Haugseng, Rune and Hebestreit, Fabian and Linskens, Sil and Nuiten, Joost},
  title         = {Lax {M}onoidal {A}djunctions, {T}wo-{V}ariable {F}ibrations and the {C}alculus of {M}ates},
  year          = {2024},
  archiveprefix = {arXiv},
  eprint        = {2011.08808},
  primaryclass  = {math.CT},
  url           = {https://arxiv.org/abs/2011.08808},
}

@Article{BD95,
  author   = {Baez, John C. and Dolan, James},
  journal  = {Journal of Mathematical Physics},
  title    = {Higher‐{D}imensional {A}lgebra and {T}opological {Q}uantum {F}ield {T}heory},
  year     = {1995},
  issn     = {0022-2488},
  month    = {11},
  number   = {11},
  pages    = {6073-6105},
  volume   = {36},
  doi      = {10.1063/1.531236},
  eprint   = {https://pubs.aip.org/aip/jmp/article-pdf/36/11/6073/19188806/6073_1_online.pdf},
  url      = {https://doi.org/10.1063/1.531236},
}

@Article{BJSS21,
  author    = {Brochier, Adrien and Jordan, David and Safronov, Pavel and Snyder, Noah},
  journal   = {Algebraic {\&}; Geometric Topology},
  title     = {Invertible {B}raided {T}ensor {C}ategories},
  year      = {2021},
  issn      = {1472-2747},
  month     = aug,
  number    = {4},
  pages     = {2107--2140},
  volume    = {21},
  doi       = {https://doi.org/10.2140/agt.2021.21.2107},
  publisher = {Mathematical Sciences Publishers},
}

\end{document}